\documentclass[11pt,a4paper]{amsart}

\RequirePackage[OT1]{fontenc}
\RequirePackage{amsthm,amsmath}
\RequirePackage[numbers]{natbib}
\RequirePackage[colorlinks,citecolor=blue,urlcolor=blue]{hyperref}

\usepackage{graphicx}
\usepackage{enumitem}
\usepackage{tikz}

\usepackage{fullpage}

\newtheorem{theorem}{Theorem}
\newtheorem{lemma}[theorem]{Lemma}
\newtheorem*{lemma*}{Lemma}

\newtheorem{corollary}[theorem]{Corollary}

\newtheorem{definition}[theorem]{Definition}
\newtheorem{remark}[theorem]{Remark}

\usepackage[utf8]{inputenc}
\usepackage[english]{babel}
\usepackage{latexsym}
\usepackage{amssymb}
\usepackage{amsmath}

\newcommand{\Z}{\mathbb{Z}}
\newcommand{\R}{\mathbb{R}}
\newcommand{\C}{\mathbb{C}}
\newcommand{\E}{\mathbb{E}}
\newcommand{\Prob}{\mathbb{P}}

\renewcommand{\Re}{\mathrm{Re} \,}
\renewcommand{\Im}{\mathrm{Im} \,}

\numberwithin{equation}{section}
\numberwithin{theorem}{section}

\author{Christian Webb}
\address{Department of mathematics and systems analysis, Aalto University, P.O.
Box 11000, 00076 Aalto, Finland}
\email{christian.webb@aalto.fi}

\author{Mo Dick Wong}
\address{{Statistical Laboratory, DPMMS, University of Cambridge. Wilberforce Rd. Cambridge CB3 0WB, United Kingdom.}}
\email{{mdw46@cam.ac.uk}}

\title{On the moments of the characteristic polynomial of a Ginibre random matrix}

\begin{document}

\begin{abstract}
In this article we study the large $N$ asymptotics of complex moments of the absolute value of the characteristic polynomial of a $N\times N$ complex Ginibre random matrix with the characteristic polynomial evaluated at a point in the unit disk. More precisely, we calculate the large $N$ asymptotics of $\E|\det(G_N-z)|^{\gamma}$, where $G_N$ is a $N\times N$ matrix whose entries are i.i.d and distributed as $N^{-1/2}Z$, $Z$ being a standard complex Gaussian, $\mathrm{Re}(\gamma)>-2$, and $|z|<1$. This expectation is proportional to the determinant of a complex moment matrix with a symbol which is supported in the whole complex plane and has a Fisher-Hartwig type of singularity: $\det(\int_\C w^{i}\overline{w}^j |w-z|^\gamma e^{-N|w|^{2}}d^2 w)_{i,j=0}^{N-1}$. We study the asymptotics of this determinant using recent results due to Lee and Yang concerning the asymptotics of orthogonal polynomials with respect to the weight $|w-z|^\gamma e^{-N|w|^2}d^2 w$ \cite{LY} along with differential identities familiar from the study of asymptotics of Toeplitz and Hankel determinants with Fisher-Hartwig singularities \cite{DIK1,DIK2,Krasovsky}. To our knowledge, even in the case of one singularity, the asymptotics of the determinant of such a moment matrix whose symbol has support in a two-dimensional set and a Fisher-Hartwig singularity, have been previously unknown.
\end{abstract}

\maketitle

\section{Introduction and main result}\label{sec:intro}

The goal of this article is to study the large $N$ asymptotics of moments of the absolute value of the characteristic polynomial of a $N\times N$ complex Ginibre random matrix, with the characteristic polynomial evaluated at a fixed point in the unit disk. More precisely, we prove the following result:

\begin{theorem}\label{th:main}
Let $G_N$ be a $N\times N$ complex Ginibre random matrix $($i.e. its entries are i.i.d. and distributed as $N^{-1/2}Z$, where $Z$ is a standard complex Gaussian$)$, $\mathrm{Re}(\gamma)>-2$, and $z\in \C$ with $|z|<1$. Then as $N\to\infty$

\begin{align*}
\E |\det(G_N-z)|^\gamma=(1+o(1))N^{\frac{\gamma^2}{8}}e^{\frac{\gamma}{2}N(|z|^2-1)}\frac{(2\pi)^{\frac{\gamma}{4}}}{G(1+\frac{\gamma}{2})},
\end{align*}

\noindent where $G$ is the Barnes G-function, and the error is uniform in $\gamma$ when restricted to a compact subset of $\lbrace \gamma\in\C: \mathrm{Re}(\gamma)>-2\rbrace$ and uniform in $z\in \lbrace w\in \C: r\leq |w|\leq R\rbrace$ with fixed $0<r\leq R<1$.
\end{theorem}

In the remainder of this introduction, we'll briefly discuss some motivation and interpretations of this result as well as give an outline of the rest of the article.

\subsection{Motivation -- moment matrices with Fisher-Hartwig singularities and random geometry} In addition to the direct application of giving information about the spectrum of the matrix $G_N$, understanding moments of the form $\E\prod_{j=1}^k |\det(G_N-z_j)|^{\gamma_j}$ is interesting due to connections to problems in various areas of mathematics. Let us first point out that if one were considering the case where $G_N$ was replaced by a Haar distributed unitary matrix (the circular unitary ensemble), such moments can be expressed as Toeplitz determinants whose symbol has so-called Fisher-Hartwig singularities. The large $N$ asymptotics of such determinants has a rather long and interesting history  -- see e.g. \cite{DIK1,DIK3,DIK2} for background and recent results concerning the problem. In the case where the matrix $G_N$ is replaced by a random Hermitian matrix such as a GUE matrix, such asymptotics have again been successfully studied through a connection to the asymptotics of Hankel determinants with Fisher-Hartwig singularities -- see e.g. \cite{Krasovsky,BWW}. 

As we will recall in Section \ref{sec:ginortho}, also moments of the form $\E\prod_{j=1}^k |\det(G_N-z_j)|^{\gamma_j}$ can be expressed in terms of determinants of moment matrices, but now of the form $\det(\int_{\C} w^i\overline{w}^j \prod_{l=1}^k |w-z_l|^{\gamma_l} e^{-N|w|^2}d^2 w)_{i,j=0}^{N-1}$. Despite the success in the case of Haar distributed unitary matrices and random Hermitian matrices, to our knowledge, there are virtually no results concerning the asymptotics of determinants of such ``fully complex" moment matrices with Fisher-Hartwig singularities (though we refer to \cite[Corollary 2]{FR}, where a representation of even integer moments of the characteristic polynomial in terms of matrix hypergeometric functions is obtained, as well as \cite{FK}, where a slightly different approach is taken for studying even integer moments of characteristic polynomials of complex random matrices). From this point of view, Theorem \ref{th:main} can be seen as a first step in the direction of a Fisher-Hartwig formula for such two-dimensional symbols.

Further motivation for Theorem \ref{th:main} comes from random geometry. In \cite{RiVi}, Rider and Vir\'ag proved a central limit theorem for linear statistics of the Ginibre ensemble (i.e. for $\mathrm{Tr}(f(G_N))$ for suitable functions $f$) and pointed out that this is roughly equivalent to $\log |\det (G_N-z)|-\E\log |\det (G_N-z)|$ converging to a variant of the Gaussian free field in a suitable sense. The limiting object here can be understood as a random generalized function which is formally a Gaussian process whose correlation kernel is $-\frac{1}{2}\log |z-w|$ for $z,w$ in the unit disk. Such random generalized functions have recently been discovered to be closely related to conformally invariant SLE-type random curves as well as the scaling limits of random planar maps -- see e.g. \cite{AJKS,Berestycki2,DKRV,MS,Sheffield}. 

In this connection between the Gaussian free field and random geometry, an important role is played by the so-called Liouville measure. This is a random measure which can formally be written as the exponential of the Gaussian free field. While the Gaussian free field is a random generalized function and exponentiating it is an operation one cannot naively perform, there is a framework for making rigorous sense of such objects. This framework is known as Gaussian multiplicative chaos and is a type of renormalization procedure to define this exponential. The original ideas of the theory go back to Kahane \cite{Kahane}, but we also refer the interested reader to the extensive review of Rhodes and Vargas \cite{RhVa} as well as the concise and elegant approach of Berestycki \cite{Berestycki1} for proving existence and uniqueness of the measure. 

Thus motivated by the central limit theory of Rider and Vir\'ag, a natural question is whether multiplicative chaos measures can be constructed from the characteristic polynomials of the Ginibre ensemble and can the limiting measure be connected to these objects appearing in random geometry. Recently, multiplicative chaos measures have been constructed from characteristic polynomials of random matrices in the setting of random unitary and random Hermitian matrices -- see \cite{BWW,LOS,Webb}. What one would expect from these results is that $\frac{|\det(G_N-z)|^\gamma}{\E |\det(G_N-z)|^\gamma}d^2 z$ converges in law to a multiplicative chaos measure as $N\to\infty$. Moreover, a central question in \cite{BWW,LOS,Webb} is to have precise asymptotics for quantities corresponding to $\E \prod_{j=1}^k |\det(G_N-z_j)|^{\gamma_j}$, so Theorem \ref{th:main} is a first step in this direction as well.

\subsection{Interpretation and comments about Theorem \ref{th:main}} We now make a few brief comments about Theorem \ref{th:main}. First of all, we point out the following immediate corollary of Theorem \ref{th:main}: if one normalizes the logarithm of the characteristic polynomial suitably, then it converges in law to standard Gaussian. More precisely, we have the following result.

\begin{corollary}\label{co:clt}
For any fixed $z\in \C$ with $|z|<1$, 

\begin{equation*}
\frac{1}{\frac{1}{2}\sqrt{\log N}}\left[\log |\det(G_N-z)|-\frac{1}{2}N(|z|^2-1)\right]\stackrel{d}{\to}N(0,1),
\end{equation*}

\noindent as $N\to\infty$. Here $N(0,1)$ denotes the standard Gaussian distribution.

\end{corollary}
\noindent To see this, note that if we write 

\begin{equation*}
X_N(z)=\frac{1}{\frac{1}{2}\sqrt{\log N}}\left[\log |\det(G_N-z)|-\frac{1}{2}N(|z|^2-1)\right],
\end{equation*}

\noindent then Theorem \ref{th:main} applied to the case $\gamma=2it/\sqrt{\log N}$ (uniformity in $\gamma$ plays an important role here) can be written as 

$$
\E e^{it X_N(z)}=(1+o(1))e^{-\frac{t^2}{2}}
$$

\noindent for each $t\in \R$. This of course implies the claim. We omit further details. Such results are typical in many random matrix models (see e.g. \cite{KS}), and may well be known for the Ginibre ensemble through other methods, though we do not know of a reference.

From our point of view, the reason to restrict to $|z|<1$ is that this is a more interesting case than $|z|>1$: one should expect from \cite{RiVi}, that for each $z\in \C$ for which $|z|>1$, $\log |\det (G_N-z)|-\E\log |\det(G_N-z)|$ converges in law to a real valued Gaussian random variable  -- there should be no $N^{\gamma^2/8}$ appearing in this case. We expect that this could be proven using a similar approach as the one we take here (using the results of \cite{LY} with $|z|>1$), but we do not explore this further. Note that another reason to distinguish between $|z|<1$ and $|z|>1$ is that in our normalization, the unit disk is the support of the equilibrium measure for the Ginibre ensemble, so it is the set where the eigenvalues should accumulate in the large $N$ limit.

We also point out that Theorem \ref{th:main} is easy to justify on a heuristic level. Indeed, proving this result for $z=0$ is very simple, as the relevant orthogonal polynomials can be calculated explicitly (see Lemma \ref{le:ortop} for the definition and importance of the orthogonal polynomials). To heuristically justify our result for $z\neq 0$, we point out that from \cite{RiVi}, one might expect that $\log |\det (G_N-z)|-\E\log |\det (G_N-z)|$ is a stationary stochastic process inside the unit disk (recall that formally this converged to a Gaussian process with translation invariant covariance), which would suggest that in Theorem \ref{th:main}, the only $z$-dependent contribution can come from $\E\log |\det (G_N-z)|$. Using e.g. \cite[Theorem 2.1]{AHM}, one would expect that 

\begin{align*}
\E\log|\det(G_N-z)|&=N \int_{|w|<1}\log |w-z|\frac{d^2w}{\pi}+\frac{1}{8\pi}\int_{|w|<1}\Delta_w \log |w-z| d^2 w+o(1)\\
&=\frac{N}{2}(|z|^2-1)+\frac{1}{4}+o(1),
\end{align*}

\noindent which suggests that $\E|\det(G_N-z)|^\gamma=\E|\det(G_N)|^\gamma e^{\frac{\gamma}{2}N |z|^2}(1+o(1))$. This is indeed true by Theorem \ref{th:main}.

Finally based on the analogy with the case of random Hermitian matrices from \cite{Krasovsky,BWW} as well as the CLT of Rider and Vir\'ag \cite{RiVi} (and that from \cite{AHM}), it would be natural to expect that a more general Fisher-Hartwig formula exists also for the Ginibre ensemble. We expect that the correct formulation would be the following: let $z_j$ be distinct fixed points in the unit disk, $\mathrm{Re}(\gamma_j)>-2$ for all $j=1,...,k$, and $f:\C\to \R$ smooth enough with compact support in the unit disk (for simplicity), then

\begin{align*}
\E e^{\mathrm{Tr} f(G_N)}\prod_{j=1}^k |\det(G_N-z_j)|^{\gamma_j}&=(1+o(1))e^{N\int_{|z|<1} f(z)\frac{d^2z}{\pi}+\frac{1}{8\pi}\int_{|z|<1}|\nabla f(z)|^2 d^2 z-\sum_{j=1}^k \frac{\gamma_j}{2} f(z_j)}\\
&\quad \times \prod_{j=1}^k N^{\frac{\gamma_j^2}{8}}e^{\frac{N}{2}\gamma_j(|z_j|^2-1)}\frac{(2\pi)^{\frac{\gamma_j}{4}}}{G(1+\frac{\gamma_j}{2})}\prod_{i<j}|z_i-z_j|^{-\frac{\gamma_i\gamma_j}{2}}.
\end{align*}

\noindent In fact, it's natural to expect that a related formula exists for more general ensembles with a regular enough confining potential. Unfortunately, we suspect that this kind of results with several singularities or non-zero $f$ are out of reach with current tools.

\subsection{Outline of the article} The outline of this article is the following. In Section \ref{sec:ginortho}, we recall how orthogonal polynomials, which are orthogonal with respect to the weight $F(w)=|w-z|^{\gamma}e^{-N|w|^2}$ (supported on the whole complex plane), are related to expectations of the form relevant to Theorem \ref{th:main}. We also recall a result of Balogh, Bertola, Lee, and McLaughlin which lets us transform orthogonality with respect to $F$ into orthogonality with respect to a weight which is supported on a contour in $\C$. In Section \ref{sec:RHPdi}, we recall how to encode these orthogonal polynomials associated to a contour into a Riemann-Hilbert problem, as well as generalize differential identities from \cite{Krasovsky,DIK1,DIK2} to facilitate efficient asymptotic analysis of the determinant of the moment matrix. Then in Section \ref{sec:RHPNasy}, we use results from \cite{LY} to solve our Riemann-Hilbert problem asymptotically. Finally in Section \ref{sec:diint}, we use our asymptotic solution of the Riemann-Hilbert problem to study the asymptotics of our differential identity, and prove Theorem \ref{th:main} by integrating the differential identity. For completeness, we also recall some basic facts about orthogonal polynomials and Riemann-Hilbert problems as well as some of the results of \cite{LY} in appendices.

\vspace{0.3cm}

{\bf Acknowledgements: } C. W. was supported by the Academy of Finland grants 288318 and 308123. M. D. Wong is supported by the Croucher Foundation Scholarship and EPSRC grant EP/L016516/1 for his PhD study at Cambridge Centre for Analysis. C. W. wishes to thank Y.V. Fyodorov for comments about the problem studied in this article.

\section{The Ginibre ensemble and orthogonal polynomials}\label{sec:ginortho}

In this section, we recall some basic facts about the complex Ginibre ensemble, such as the distribution of the eigenvalues, how expectations of suitable functions of eigenvalues of Ginibre random matrices can be expressed in terms of determinants of complex moment matrices, as well as how such questions relate to orthogonal polynomials. We also recall results from \cite{BBLM,LY} which show that the orthogonal polynomials associated to the expectation $\E|\det(G_N-z)|^\gamma$ also satisfy suitable orthogonality conditions on certain contours in the complex plane. Then in Section \ref{sec:RHPdi}, we apply these results to transform the analysis of $\E|\det(G_N-z)|^\gamma$ into a question of the asymptotic analysis of a suitable Riemann-Hilbert problem. For the convenience of the reader, we sketch proofs of some of the statements of this section in Appendix \ref{app:ortho}.

As stated in Theorem \ref{th:main}, $G_N$ is a random $N\times N$ matrix whose entries are i.i.d. and distributed as $N^{-1/2}Z$, where $Z$ is a standard complex Gaussian. We recall that the law of the eigenvalues of $G_N$ can then be expressed in the following form \cite{G}

\begin{equation}\label{eq:evlaw}
\Prob(d^2z_1,...,d^2z_N)=\frac{1}{Z_N}\prod_{1\leq i<j\leq N}|z_i-z_j|^2 \prod_{j=1}^N e^{-N|z_j|^2}d^2 z_j
\end{equation}

\noindent on $\C^N$. Here the normalizing constant $Z_N$ is

\begin{equation*}
Z_N=\pi^N\frac{\prod_{k=1}^N k!}{N^{\frac{N(N+1)}{2}}}.
\end{equation*}

\noindent We'll denote integration with respect to $\Prob(d^2z_1,...,d^2z_N)$ by $\E$ -- so we suppress the dependence on $N$ in our notation.

We now recall a Heine-Szeg\H{o}-type identity which connects the Ginibre ensemble to determinants of complex moment matrices.

\begin{lemma}\label{le:heine}
Let $F:\C\to \C$ be regular enough $($so that $\int_\C |w|^{k}|F(w)|e^{-N|w|^2}d^2w<\infty$ for all $k\geq 0)$, then 

\begin{equation*}
\E\prod_{j=1}^N F(z_j)=\frac{N!}{Z_{N}}D_{N-1}(F):=\frac{N!}{Z_{N}}\det \left(\int_\C w^i\overline{w}^j F(w) e^{-N|w|^2}d^2w\right)_{i,j=0}^{N-1}.
\end{equation*}
\end{lemma} 

\noindent This is a straightforward generalization of a corresponding identity for random Hermitian and random unitary matrices and relies on noticing that $\prod_{i<j}|z_i-z_j|^2$ in \eqref{eq:evlaw} can be written in terms of the Vandermonde determinant which then allows this determinantal representation. We omit further details.

The next fact we need is the connection between $D_{N-1}(F)$ defined in Lemma \ref{le:heine} and suitable orthogonal polynomials. To do this, let us introduce the notation 

$$D_k^{(N)}(F)=\det\left(\int_\C s^i \overline{s}^j F(s)e^{-N|s|^2}ds\right)_{i,j=0}^k$$

\noindent and if $D_{j-1}^{(N)}(F),D_j^{(N)}(F)\neq 0$, write 

\begin{equation}\label{eq:polydet}
p_j(w)=\frac{1}{\sqrt{D_{j-1}^{(N)}(F)D_{j}^{(N)}(F)}}\begin{vmatrix}
\int_\C F(s)e^{-N|s|^2}d^2s  & \cdots & \int_\C s^{j}F(s)e^{-N|s|^2}d^2s\\
\vdots &   & \vdots\\
\int_\C \overline{s}^{j-1}F(s)e^{-N|s|^2}d^2s & \cdots & \int_\C\overline{s}^{j-1}s^{j}F(s)e^{-N|s|^2}d^2s\\
1 &  \cdots & w^{j}
\end{vmatrix},
\end{equation}

\noindent where the branch of the square root is the principal one and the interpretation is that $D_{-1}^{(N)}(F)=1$ and for $j=0$, the determinant is replaced by $1$.

The following (standard) lemma demonstrates some basic orthogonality properties of the polynomials $p_j$ along with the connection between $D_{N-1}(F)$ and the leading order coefficients of $p_j$.

\begin{lemma}\label{le:ortop}
Let $F:\C\to \C$ be regular enough $($so that $\int_\C |w|^k {\color{blue}|}F(w){\color{blue}|}e^{-N|w|^2}d^2 w<\infty$ for all $k\geq 0)$ and assume that $D_{j-1}^{(N)}(F),D_j^{(N)}(F)\neq 0$. Let us also write $\chi_j$ for the coefficient of $w^j$ in $p_j(w)$ $($note that under our assumptions, this is non-zero$)$. Then for any $0\leq k\leq j$
\begin{equation}\label{eq:ortop}
\int_\C p_j(w) \overline{w}^k F(w)e^{-N|w|^2}d^2 w=\frac{\delta_{j,k}}{\chi_j}.
\end{equation}

\noindent Moreover if $D_{j}^{(N)}(F)\neq 0$ for $0\leq j\leq N-1$, then

\begin{equation}\label{eq:DNchi}
D_{N-1}(F)=\prod_{j=0}^{N-1}\chi_j^{-2}.
\end{equation}
\end{lemma}

\noindent The orthogonality condition \eqref{eq:ortop} follows easily from noticing that by linearity of the determinant, if $k <j$, the determinantal expression of $\int_\C p_j(w)\overline{w}^k F(w)e^{-N|w|^2}d^2 w$ will have two identical rows and thus vanish. For $j=k$, \eqref{eq:ortop} follows from comparing with \eqref{eq:polydet}. \eqref{eq:DNchi} follows from our definition of $D_{-1}^{(N)}(F)=1$ and the telescopic structure of the product (in particular, $\chi_j^{-2}=D_j^{(N)}(F)/D_{j-1}^{(N)}(F)$). We omit further details.

The next ingredient we shall need for our Riemann-Hilbert problem is a fact noticed in \cite{BBLM}, namely that in the special case when $F(w)=|w-z|^\gamma$, the polynomials $p_j$ from Lemma \ref{le:ortop} satisfy certain orthogonality relations on suitable contours in the complex plane as well. To simplify notation slightly, we shall first note that the law of $(z_i)_{i=1}^N$ is invariant under rotations: it follows easily from \eqref{eq:evlaw} that for fixed $\theta\in \R$, $(e^{i\theta}z_j)_{j=1}^N$ has the same law as $(z_j)_{j=1}^N$. From this it follows that $\E|\det(G_N-z)|^\gamma=\E|\det(G_N-|z|)|^\gamma$. We thus see that for Theorem \ref{th:main}, it's enough to understand the asymptotics of $\E|\det(G_N-x)|^\gamma$ for $x\in(0,1)$. To emphasize this we now restrict our attention to weights $F$ that are relevant to this expectation: we fix our notation in the following definition.

\begin{definition}\label{def:Fp}
For $x\in(0,1)$ and $\mathrm{Re}(\gamma)>-2$, let $F:\C\to \C$, $F(w)=|w-x|^\gamma$. Moreover, when they exist, $($i.e. when $D_{j-1}^{(N)}(F),D_j^{(N)}(F)\neq 0)$ let $(p_j)_{j=0}^\infty$ be the polynomials from Lemma \ref{le:ortop} associated to this $F$ and let $\chi_j$ be the coefficient of $w^j$ in $p_j(w)$ -- in our notation, we omit the dependence on $N$, $\gamma$, and $x$.
\end{definition}

\noindent The statement about orthogonality on suitable contours discovered in \cite[Lemma 3.1]{BBLM} is the following.

\begin{lemma}[Balogh, Bertola, Lee, and McLaughlin]\label{le:ortoc}
Let $\Sigma$ be a simple, smooth, and closed contour in the complex plane, and let it encircle $[0,x]$, possibly passing through $x$, but not other points of $[0,x]$, and let it be oriented in the counter-clockwise direction. Let 

\begin{equation}\label{eq:f}
f(w)=w^{-\frac{\gamma}{2}}(w-x)^{\frac{\gamma}{2}}e^{-Nxw},
\end{equation}

\noindent where the roots are according to the principal branch $($so the branch cut of $f$ is $[0,x])$. If $D_{j-1}^{(N)}(F),D_j^{(N)}(F)\neq 0$, then for $0\leq k\leq j$,

\begin{equation}\label{eq:ortoc}
\oint_\Sigma p_j(w)w^{-k}f(w)\frac{dw}{2\pi i w}=\begin{cases}
0, & k<j\\
\frac{1}{\pi}\frac{N^{1+\frac{\gamma}{2}+k}}{\Gamma(1+\frac{\gamma}{2}+k)}\frac{1}{\chi_j}, & k=j
\end{cases}.
\end{equation}
\end{lemma}

As the situation considered in \cite{BBLM} is slightly different -- for them $\gamma$ is proportional to $N$ (and real), and their result is stated for contours avoiding $x$, we sketch a proof in Appendix \ref{app:ortho}. We also point out that if $\Sigma$ were the unit circle, \eqref{eq:ortoc} would look like a basic orthogonality condition for polynomials on the unit circle. Thus (as in \cite{BBLM,LY}) it is fruitful to define a dual family of polynomials which are orthogonal to the polynomials $p_j$ with respect to the pairing coming from \eqref{eq:ortoc}. We now recall how these dual orthogonal polynomials are constructed and how their leading order coefficient is related to $\chi_j$.

\begin{lemma}\label{le:ortoq}
Let $\Sigma$ and $f$ be as in Lemma \ref{le:ortoc}.

\begin{itemize}[leftmargin=0.5cm]
\item[{\rm (i)}] Let us define for any $j\geq 0$

\begin{equation*}
\widehat{D}_j=\det\left(\oint_\Sigma w^{-(r-s)}f(w)\frac{dw}{2\pi i w}\right)_{r,s=0}^j.
\end{equation*}

\noindent Then 

\begin{equation}\label{eq:Dhat}
\widehat{D}_j=D_j^{(N)}(F)\prod_{k=0}^j\frac{1}{\pi}\frac{N^{1+\frac{\gamma}{2}+k}}{\Gamma(1+\frac{\gamma}{2}+k)}.
\end{equation}

\item[{\rm (ii)}] Assume that $D_{j-1}^{(N)}(F), D_j^{(N)}(F)\neq 0$ and define for $w\neq 0$ 

\begin{equation}\label{eq:qdef}
q_j(w^{-1})=\frac{\prod_{k=0}^j \frac{\pi \Gamma\left(\frac{\gamma}{2}+k+1\right)}{N^{\frac{\gamma}{2}+k+1}}}{\sqrt{D_{j-1}^{(N)}(F)D_j^{(N)}(F)}}\begin{vmatrix}
\oint_\Sigma f(s)\frac{ds}{2\pi i s} & \cdots & \oint_\Sigma s^{j-1}f(s)\frac{ds}{2\pi i s} & 1\\
\vdots &  & \vdots & \vdots \\
\oint_\Sigma s^{-j} f(s)\frac{ds}{2\pi i s} & \cdots & \oint_\Sigma s^{-1}f(s)\frac{ds}{2\pi i s} & w^{-j}
\end{vmatrix},
\end{equation}

\noindent where the branch of the root is the principal one. Then for $0\leq k\leq j$

\begin{equation}\label{eq:qporto}
\oint_\Sigma w^kq_j(w^{-1})f(w)\frac{dw}{2\pi i w}=\frac{\delta_{j,k}}{\chi_k}
\end{equation}

\noindent and if we write $\widehat{\chi}_j$ for the coefficient of $w^{-j}$ in $q_j(w^{-1})$, then 

\begin{equation}\label{eq:chihat}
\widehat{\chi}_j=\chi_j\frac{\pi \Gamma(1+\frac{\gamma}{2}+j)}{N^{1+\frac{\gamma}{2}+j}}.
\end{equation}
\end{itemize}

\end{lemma}

\noindent Again, we offer a sketch of a proof in Appendix \ref{app:ortho}, as such a result isn't formulated precisely in this form in \cite{BBLM,LY}. We now turn to the Riemann-Hilbert problem and the differential identity related to $D_{N-1}(F)$.

\section{The Riemann-Hilbert problem and the differential identity}\label{sec:RHPdi}

We are now in a position to encode our polynomials into a Riemann-Hilbert problem in a similar way as in \cite{BBLM,LY} as well as state our differential identity. The proof of the differential identity is a modification of those appearing in \cite{DIK1,DIK2,Krasovsky}, but as the differential identity in our case is slightly more complicated, we offer details for the proof in Appendix \ref{app:RHP}.

We begin by defining the object that will satisfy a Riemann-Hilbert problem.

\begin{definition}\label{def:Ydef}
Let $\Sigma$ be as in Lemma \ref{le:ortoc} and assume that $D_{j-2}^{(N)}(F),D_{j-1}^{(N)}(F),D_j^{(N)}(F)\neq 0$. For $w\notin \Sigma$ and $j\geq 1$, let 

\begin{equation}\label{eq:Ydef}
Y(w)=Y_j(w)=\begin{pmatrix}
\frac{1}{\chi_j}p_j(w) & \frac{1}{\chi_j}\oint_\Sigma \frac{s^{-(j-1)}p_j(s)f(s)}{s-w}\frac{ds}{2\pi i s}\\
-\chi_{j-1}w^{j-1}q_{j-1}(w^{-1}) & -\chi_{j-1}\oint_\Sigma \frac{q_{j-1}(s^{-1})f(s)}{s-w}\frac{ds}{2\pi i s}
\end{pmatrix}.
\end{equation}
\end{definition}

\noindent Note that for each $j$, $Y_j$ also depends on $N$, $x$, $\gamma$, as well as the contour $\Sigma$ we have not yet fixed, but we suppress this in our notation. 

As originally noticed by Fokas, Its, and Kitaev \cite{FIK}, such an object indeed satisfies a Riemann-Hilbert problem:

\begin{lemma}\label{le:RHP}
Let $D_{j-2}^{(N)}(F),D_{j-1}^{(N)}(F),D_j^{(N)}(F)\neq 0$. Then $Y=Y_j$ is the unique solution to the following Riemann-Hilbert problem. 

\begin{itemize}[leftmargin=0.5cm]
\item $Y:\C\setminus \Sigma\to \C^{2\times 2}$ is analytic.
\item $Y$ has continuous boundary values on $\Sigma\setminus\lbrace x\rbrace$ $($denote by $Y_+$ the limit from the side of the origin and by $Y_-$ the limit from the side of infinity$)$ and they satisfy the following jump relation$:$ for $w\in \Sigma\setminus\lbrace x\rbrace$

\begin{equation}\label{eq:Yjump}
Y_+(w)=Y_-(w)\begin{pmatrix}
1 & w^{-j}f(w)\\
0 & 1
\end{pmatrix}.
\end{equation}
\item As $w\to \infty$, 

\begin{equation}\label{eq:Yasy}
Y(w)=(I+\mathcal{O}(w^{-1}))w^{j\sigma_3}=(I+\mathcal{O}(w^{-1}))\begin{pmatrix}
w^j & 0\\
0 & w^{-j}
\end{pmatrix}
\end{equation}

\noindent where $I$ is the $2\times 2$ identity matrix and $\mathcal{O}(w^{-1})$ denotes a $2\times 2$ matrix whose entries are bounded by $|w|^{-1}$ as $w\to \infty$.

\item As $w\to x$, 

\begin{equation}\label{eq:Yatxasy1}
Y(w)=\begin{pmatrix}
\mathcal{O}(1) & \mathcal{O}(1)+\mathcal{O}\left(|w-x|^{\mathrm{Re}(\gamma/2)}\right)\\
\mathcal{O}(1) & \mathcal{O}(1)+\mathcal{O}\left(|w-x|^{\mathrm{Re}(\gamma/2)}\right)
\end{pmatrix}.
\end{equation}

\end{itemize}
\end{lemma}

\begin{remark}\label{rem:yatxasy}
As we will see later, actually $Y(w)$ converges to a finite limit as $w\to x$ from $\mathrm{Int}(\Sigma)$. This is important for our differential identity. Nevertheless, as $w\to x$ from $\mathrm{Ext}(\Sigma)$, $Y(w)$ remains unbounded if $\mathrm{Re}(\gamma)<0$.
\end{remark}

\noindent The proof is essentially standard -- uniqueness of a solution follows from Liouville's theorem (along with some standard arguments about a possible singularity at $x$ not being strong enough to be a pole due to the condition $\mathrm{Re}(\gamma)>-2$), the jump conditions from the Sokhotski-Plemelj theorem, and the asymptotic behavior at infinity from the orthogonality conditions \eqref{eq:ortoc} and \eqref{eq:qporto}. The continuity of the boundary values along with the asymptotic behavior at $x$ follow from basic properties of boundary values of the Cauchy transform -- see e.g. \cite[\S 19 and \S 33]{musk}.  We omit further details of the proof and refer to e.g. \cite{deift,kuijlaars}.

As we have seen in Lemma \ref{le:ortop}, if $D_{j}^{(N)}(F)\neq 0$ for $j\leq N-1$, one way to obtain asymptotics for $D_{N-1}(F)$ would be to obtain good asymptotics for $\chi_j$ for all $j\leq N-1$  (or $Y_j$ for all $j\leq N-1$), which would suggest that one would need to solve the above RHP for all $j\leq N-1$.  Due to a differential identity we now describe, it's enough for us to only solve the problem for $Y_N$ and $Y_{N+1}$.

\begin{lemma}\label{le:DI}
Let us write $D_{N-1}(F;\gamma)=D_{N-1}(F)$ $($where again $F(w)=|w-x|^\gamma)$ and assume that $D_{j}^{(N)}(F)\neq 0$ for $j\leq N+1$.  Let us also write $\kappa_j$ for the coefficient of $w^{j-1}$ in $p_j(w)$. Then for $\mathrm{Re}(\gamma)>-2$ 

\begin{align*}
& \partial_\gamma \log D_{N-1}(F;\gamma)\\
&=-\left(N+\frac{\gamma}{2}\right)\frac{\partial_\gamma\widehat{\chi}_N}{\widehat{\chi}_N}+\frac{\gamma}{2}x^N\frac{\partial q_N(x^{-1})}{\partial\gamma}\lim_{z \to x} \oint_\Sigma p_N(w)w^{-N+1}\frac{f(w)}{w-z}\frac{dw}{2\pi i w} -Nx\frac{p_{N+1}(0)}{\chi_{N+1}}\frac{\partial_\gamma q_N(0)}{\widehat{\chi}_N}\\
&\quad -N\frac{\partial_\gamma \chi_N}{\chi_N}-\frac{\gamma}{2}\frac{\partial p_N(x)}{\partial\gamma}x\lim_{z \to x} \oint_\Sigma \frac{q_N(w^{-1})f(w)}{w-z}\frac{dw}{2\pi i w}+Nx\left(\frac{\partial_\gamma \kappa_N}{\chi_N}-\frac{\partial_\gamma \chi_N}{\chi_N}\frac{\kappa_{N+1}}{\chi_{N+1}}\right)\\
&\quad +\partial_\gamma \sum_{j=0}^{N-1}\log \frac{\Gamma\left(\frac{\gamma}{2}+j+1\right)}{N^{\frac{\gamma}{2}}}
\end{align*}

\noindent where all the limits on the RHS should be interpreted as being taken along any sequence in $\mathrm{Int}(\Sigma) \setminus [0,x]$ tending to $x$.

\end{lemma}
\noindent Note that as $Y$ has no singularities on $(0,x)$, it would be natural to expect that one could take the sequence to be on this interval as well. Our proof does involve objects with branch cuts on $[0,x]$ and the proof would become slightly more involved if we wished to allow points on $[0,x]$ as well. For simplicity, we thus focus on sequences in $\mathrm{Int}(\Sigma)\setminus [0,x]$.

We give a proof of this differential identity in Appendix \ref{app:RHP}. One can easily check that all of the quantities here can be expressed in terms of $Y_N$ and $Y_{N+1}$ -- e.g. $\chi_N\widehat{\chi}_N=-Y_{N+1,21}(0)$, from which one can solve $\chi_N$. See Section \ref{sec:diint} for further details. We now move onto the asymptotic analysis of $Y_N$ and $Y_{N+1}$ by solving their RHPs. 

\section{Solving the Riemann-Hilbert problem for \texorpdfstring{$Y_N$}{YN} asymptotically}\label{sec:RHPNasy}

In this section we recall from \cite{LY} the asymptotic solution of the RHP for $Y_N$. In fact, we'll consider a minor generalization of their situation where we study the asymptotics of $Y_{N+k}$, where $k$ is a fixed integer -- for our differential identity, we only need $k=0$ and $k=1$. Again we offer details of the argument in Appendix \ref{app:RHPasy} since the question is slightly different from that in \cite{LY}. For intuition and further discussion concerning the approach, we refer to \cite{LY} and references therein.

As typical in this type of Riemann-Hilbert problems, using approximate problems which can be solved explicitly, we will transform this problem into a "small-norm" problem which can be solved asymptotically in terms of a Neumann-series. The solutions to the approximate problems are called parametrices, and we will need two of them: one close to the point $x$, and one far away from it. The one close to $x$ is called the local parametrix and the one far from it is the global parametrix. We begin with a transformation that normalizes our problem at infinity and enables ``opening lenses", then we recall from \cite{LY} the global and local parametrices relevant to us. Finally we will consider the solution of the small norm problem. Throughout this section, we will implicitly be assuming that the RHP for $Y$ is solvable, or that the relevant orthogonal polynomials exist, unless otherwise stated.

\subsection{Transforming the problem}

The goal of the transformation procedure is to have a RHP which is normalized at infinity (the sought function converges to the identity matrix as $w\to\infty$) and for which the jump matrix is close to the identity as $N\to\infty$. This allows formulating the problem in terms of a certain singular integral equation which can be solved in terms of a suitable Neumann-series. We begin by normalizing the function at infinity. To do this, let us write $\mathrm{Ext}(\Sigma)$ for the unbounded component of $\C\setminus \Sigma$ and $\mathrm{Int}(\Sigma)$ for the bounded one (recall that we still have not fixed $\Sigma$, but we will do this shortly), and define

\begin{equation}\label{eq:lgdef}
\ell=\log x-x^2\qquad \mathrm{and} \qquad g(w)=\begin{cases}
\log w, & w\in \mathrm{Ext}(\Sigma)\\
\ell+xw, & w\in \mathrm{Int}(\Sigma)
\end{cases}.
\end{equation}

\noindent As we are only giving a brief overview of the approach of \cite{LY}, we refer to \cite{BBLM,LY} for a discussion of why $\ell$ and $g$ are chosen so. Throughout this section, we will be working with $Y=Y_{N+k}$ and we will drop for now the index $N+k$ from our notation. We then define 

\begin{equation}\label{eq:Tdef}
T(w)=e^{-(N+k)\frac{\ell}{2}\sigma_3}Y(w)e^{-(N+k)g(w)\sigma_3}e^{(N+k)\frac{\ell}{2}\sigma_3}.
\end{equation}

\noindent Note that from the asymptotic behavior of $Y$ at infinity, namely \eqref{eq:Yasy}, and our choice of $g$ in $\mathrm{Ext}(\Sigma)$, we see that $T(w)=I+\mathcal{O}(w^{-1})$ as $w\to\infty$. 

Let us next fix the contour $\Sigma$. Let

\begin{align}\label{eq:Sigmadef}
\Sigma&=\lbrace w\in \C: \mathrm{Re}(xw+\ell-\log w)=0, \mathrm{Re}(w)\leq x\rbrace\\
\notag &=\lbrace u+iv\in \C: u^2+v^2=x^2e^{2x(u-x)}, u\leq x\rbrace. 
\end{align}

\noindent The point of choosing our jump contour to be this one will be evident shortly as we'll perform another transformation which will result in a jump matrix close to the identity when off of $\Sigma\cup[0,x]$.  Before going into our next transformation, we point out the following fact (see also \cite[Lemma 4]{LY}).

\begin{lemma}\label{le:Sigmasign}
For each $x\in(0,1)$, $\Sigma$ is a smooth, simple closed loop inside the unit disk. It encircles $[0,x]$ $($passing through $x$, but not other points$)$, and 

\begin{align*}
\mathrm{Re}(xw+\ell-\log w)\begin{cases}
>0, & w\in\mathrm{Int}(\Sigma)\\
<0, & w\in \lbrace s\in \mathrm{Ext}(\Sigma): |s|\leq 1\rbrace
\end{cases},
\end{align*}

\noindent moreover $\sup_{|w|=1}\mathrm{Re}(xw+\ell-\log w)<0$ for all $x\in (0,1)$.

\end{lemma}

\noindent Note that in particular, $\Sigma$ satisfies the conditions of Lemma \ref{le:ortoc}. The proof is given in Appendix \ref{app:RHPasy}.

Our next transformation allows us to perform a Deift-Zhou non-linear steepest descent-type argument by opening lenses. Our lens will now essentially be the unit circle combined with the interval $[0,x]$. We define

\begin{align}\label{eq:Sdef}
S(w)=\begin{cases}
T(w), & |w|>1\\
T(w)\begin{pmatrix}
1 & 0\\
w^{\frac{\gamma}{2}}(w-x)^{-\frac{\gamma}{2}}e^{-kxw} e^{(N+k)(xw+\ell-\log w)} & 1
\end{pmatrix}, & w\in \lbrace s\in \mathrm{Ext}(\Sigma): |s|<1\rbrace\\
T(w)\begin{pmatrix}
1 & 0\\
-w^{\frac{\gamma}{2}}(w-x)^{-\frac{\gamma}{2}}e^{-kxw}e^{-(N+k)(xw+\ell-\log w)} & 1
\end{pmatrix},  & w\in \mathrm{Int}(\Sigma)\setminus [0,x]
\end{cases},
\end{align}

\noindent where as before, the roots are according to the principal branch.

\begin{figure}[h!]
\begin{center}
\begin{tikzpicture}[scale = 2.5, every node/.style={scale=1}]
\draw[->] (-2,0) -- (2, 0) node[right] {$\mathrm{Re}(w)$};
\draw[->] (0, -1.2) -- (0, 1.2) node[above] {$\mathrm{Im}(w)$};

\draw[thick, ->] (1, 0) to [out = 90, in = 0] (0,1) node[above right] {$-$} node [below right]{$+$};
\draw[thick, ->] (-1, 0) to [out = -90, in = -180] (0, -1) node[above left] {$+$} node [below left] {$-$};
\draw[thick] (0,0) node[below left] {$0$} circle [radius = 1];

\node at (1.15, 0.8) {$\{|w| = 1\}$};

\draw[thick, ->] (0,0) -- (0.7, 0) (0,0) -- (0.4, 0) node[above]{$+$} node[below]{$-$};

\draw[thick, <-, domain= 0.2:0.7, samples=40] plot ({\x}, {sqrt(0.49*exp(1.4*(\x-0.7))-\x*\x)});
\draw[thick, domain=-0.338392259:0.7, samples=40] plot ({\x}, {sqrt(0.49*exp(1.4*(\x-0.7))-\x*\x)});
\draw[thick, ->, domain= -0.338392259:0.2, samples=40] plot ({\x}, {-sqrt(0.49*exp(1.4*(\x-0.7))-\x*\x)});
\draw[thick, domain=-0.338392259:0.7, samples=40] plot ({\x}, {-sqrt(0.49*exp(1.4*(\x-0.7))-\x*\x)});
\node at (0.25, 0.55) {$-$};
\node at (0.25, 0.35) {$+$};
\node at (0.2, -0.55) {$-$};
\node at (0.2, -0.35) {$+$};
\node at (-0.35, 0.35) {$\Sigma$};

\draw[fill = black] (0.7, 0) circle [radius = 0.02] node[below right]{$x$};
\draw[fill = black] (0, 0) circle [radius = 0.02];
\end{tikzpicture}
\end{center}
\caption{\label{fig:lenses}$S$-RHP and the opening of lenses.}
\end{figure}
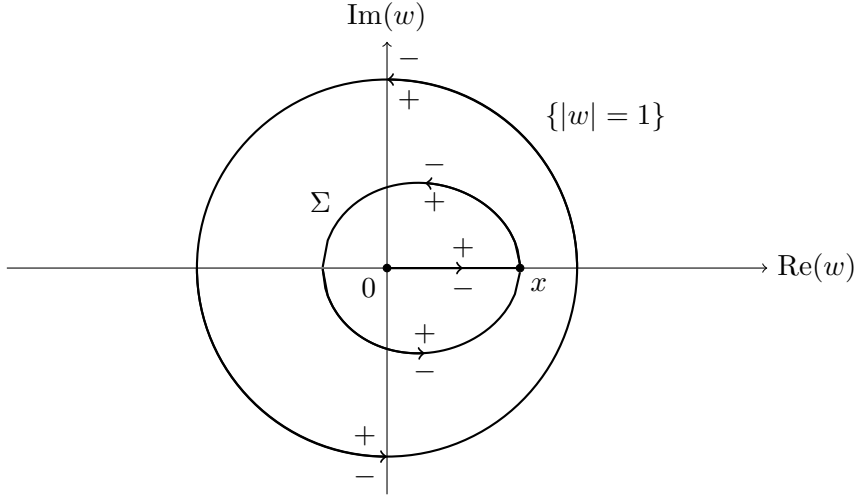

We now describe the RHP satisfied by this function. 

\begin{lemma}\label{le:SRHP}
Let $C=[0,x]\cup \Sigma\cup \lbrace w\in \C:|w|=1\rbrace$. Orient $[0,x]$ from $0$ to $x$ so that the $+$ side of the interval is the upper half plane. Orient the unit circle so that the inside of the circle is the $+$ side of the contour, and orient $\Sigma$ in the counter-clockwise direction $($i.e. we let the $+$ side of the contour be the side of the origin and the $-$ side of the contour be the side of infinity$)$. Then $S$ satisfies the following Riemann-Hilbert problem.

\begin{itemize}[leftmargin=0.5cm]
\item $S:\C\setminus C\to \C^{2\times 2}$ is analytic.
\item $S$ has continuous boundary values on $C\setminus \lbrace x,0\rbrace$, and these satisfy the following jump conditions:

\begin{equation}\label{eq:Sjump1}
S_+(w)=S_-(w)\begin{pmatrix}
1 & 0\\
w^{\frac{\gamma}{2}}(w-x)^{-\frac{\gamma}{2}}e^{-kxw}e^{(N+k)(xw+\ell-\log w)} & 1
\end{pmatrix}, \qquad |w|=1,
\end{equation}

\begin{equation}\label{eq:Sjump2}
S_+(w)=S_-(w)\begin{pmatrix}
0 & e^{kxw}(w-x)^{\frac{\gamma}{2}} w^{-\frac{\gamma}{2}}\\
-e^{-kxw}(w-x)^{-\frac{\gamma}{2}}w^{\frac{\gamma}{2}} & 0
\end{pmatrix}, \qquad w\in \Sigma\setminus \lbrace x\rbrace, 
\end{equation}

\noindent and 

\begin{equation}\label{eq:Sjump3}
S_+(w)=S_-(w)\begin{pmatrix}
1 & 0\\
2i \sin \frac{\pi \gamma}{2}|w|^{\gamma/2}|w-x|^{-\frac{\gamma}{2}} e^{-kxw}e^{-(N+k)(xw+\ell-\log w)} & 1
\end{pmatrix}, \qquad w\in (0,x).
\end{equation}
\item As $w\to 0$, $S(w)$ is bounded $($actually $S(0)$ exists$)$ and as $w\to x$ $($off of $C)$,

\begin{equation}\label{eq:Sasy}
S(w)=\begin{pmatrix}
\mathcal{O}(1)+\mathcal{O}\left(|w-x|^{-\frac{\mathrm{Re}(\gamma)}{2}}\right) & \mathcal{O}(1)+\mathcal{O}\left(|w-x|^{\frac{\mathrm{Re}(\gamma)}{2}}\right)\\
\mathcal{O}(1)+\mathcal{O}\left(|w-x|^{-\frac{\mathrm{Re}(\gamma)}{2}}\right) & \mathcal{O}(1)+\mathcal{O}\left(|w-x|^{\frac{\mathrm{Re}(\gamma)}{2}}\right) 
\end{pmatrix}.
\end{equation}

\item As $w\to \infty$, $S(w)=I+\mathcal{O}(w^{-1})$.
\end{itemize}
\end{lemma}

\noindent The proof is in Appendix \ref{app:RHPasy}.

Our next task is to find the approximate solutions. The first one corresponds to focusing on a problem where we only consider the jump condition \eqref{eq:Sjump2} (the global parametrix), while the second one approximates the RHP close to the point $x$ (the local parametrix) as well as approximately matches the global solution on the boundary of a small neighborhood of the point $x$.

\subsection{The global parametrix}\label{sec:global}

Here we first look for a function $P^{(\infty)}:\C\setminus \Sigma\to \C^{2\times 2}$ which satisfies the jump condition \eqref{eq:Sjump2} and is normalized at infinity. We simply mention that one can easily check that the function 

\begin{equation}\label{eq:global}
P^{(\infty)}(w)=\begin{cases}
\begin{pmatrix}
w^{\frac{\gamma}{2}}(w-x)^{-\frac{\gamma}{2}}  & 0\\
0 & w^{-\frac{\gamma}{2}}(w-x)^{\frac{\gamma}{2}}
\end{pmatrix}, & w\in \mathrm{Ext}(\Sigma)\\
\begin{pmatrix}
0 & e^{kxw}\\
-e^{-kxw} & 0
\end{pmatrix}, & w\in \mathrm{Int}(\Sigma)
\end{cases}
\end{equation}

\noindent satisfies these conditions.

If we were to take this as our global parametrix, we would obtain a small norm problem for $\mathrm{Re}(\gamma)<2$ and it could be solved as an expansion in $N^{\frac{\mathrm{Re}(\gamma)}{2}-1}$, which would be sufficient for small enough $\gamma$, but as we are interested also in larger $\mathrm{Re}(\gamma)$, this parametrix is not good enough for us. It turns out that for our differential identity, we'll need to adjust the global parametrix depending on the size of $\gamma$, and in fact we need to define a sequence of global parametrices. The way we'll shortly define this sequence is as

$$
\widehat{P}^{(\infty,r)}(w)=\begin{pmatrix}
1 & h_r(w;\gamma)\\
0 & 1
\end{pmatrix}P^{(\infty)}(w),
$$

\noindent where $h_r$ is a Laurent polynomial of the form $\sum_{j=0}^r h_{j,r}(\gamma) (w-x)^{-j-1}$, with $h_{j,r}$ being some suitable coefficients that need to be chosen to ensure that the local parametrix we construct in the next section has the correct behavior at $x$. This will eventually result in a small norm problem which will yield an expansion in $N^{\frac{\gamma}{2}-r-1}$. Note that for any Laurent polynomial $h_r$, $\widehat{P}^{(\infty,r)}$ will have the same jump structure as $P^{(\infty)}$ -- namely it satisfies \eqref{eq:Sjump2}, though the behavior at $x$ will be different.

We will now introduce some notation to be able to make the relevant definition of $h_r$ and in the following section, where we discuss the local parametrix, it will hopefully become more apparent why such a definition is required.

Consider $\zeta:\C\setminus (-\infty,0]\to \C$,

\begin{equation}\label{eq:zeta}
\zeta(w)=-(N+k)(xw-\log w+\ell),
\end{equation}

\noindent where the branch is the principal one. We can now define our functions $h_r$.

\begin{definition}\label{def:h}
For $r\geq 0$ define $h_r(w;\gamma)=\sum_{j=0}^{r}h_{j,r}(\gamma)(w-x)^{-j-1}$ to be the unique function of such form that 

\begin{equation}\label{eq:hcond}
w\mapsto h_r(w;\gamma)-e^{kxw}w^{\gamma/2}(w-x)^{-\gamma/2}\zeta(w)^{\gamma/2}\sum_{j=0}^{r}\frac{1}{\Gamma(\frac{\gamma}{2}-j)}\zeta(w)^{-j-1}
\end{equation}

\noindent is analytic in some $(N$-independent$)$ neighborhood of $x$. Above, the branch of the root is again the principal one. Also define for $w\notin\Sigma$, 

\begin{equation}\label{eq:Phat}
\widehat{P}^{(\infty,r)}(w)=\begin{pmatrix}
1 & h_r(w;\gamma)\\
0 & 1
\end{pmatrix}P^{(\infty)}(w).
\end{equation}

\end{definition}

Note that this definition of $h_r$ makes sense: as $\zeta$ has an order one zero at $x$, $w\mapsto w^{\gamma/2}(w-x)^{-\gamma/2}\zeta(w)^{\gamma/2}$ is analytic in some ($N$-independent) neighborhood of $x$, so 

$$
e^{kxw}w^{\gamma/2}(w-x)^{-\gamma/2}\zeta(w)^{\gamma/2}\sum_{j=0}^{r}\frac{1}{\Gamma(\frac{\gamma}{2}-j)}\zeta(w)^{-j-1}
$$

\noindent is a sum of a degree (at most) $r+1$ Laurent polynomial and an analytic function, so by subtracting the poles, one is left with an analytic function. 

We will also need some simple properties of the function $h_r$ and we record them in the following lemma.

\begin{lemma}\label{le:h}
The functions $\gamma \mapsto h_r(0,\gamma)$ and 

$$
\gamma\mapsto\lim_{w\to x} \left[h_r(w,\gamma)-e^{kxw}w^{\gamma/2}(w-x)^{-\gamma/2}\zeta(w)^{\gamma/2}\sum_{j=0}^{r}\frac{1}{\Gamma(\frac{\gamma}{2}-j)}\zeta(w)^{-j-1}\right]
$$

\noindent are analytic functions of $\gamma$ in $\lbrace \gamma\in \C:\mathrm{Re}(\gamma)>-2\rbrace$. Moreover, we have the bounds 

\begin{equation}\label{eq:hbound1}
h_r(0,\gamma)=\mathcal{O}\left(N^{\frac{\mathrm{Re}(\gamma)}{2}-1}\right),
\end{equation}

\begin{equation}\label{eq:hbound2}
\partial_\gamma h_r(0,\gamma)=\mathcal{O}\left(\log N N^{\frac{\mathrm{Re}(\gamma)}{2}-1}\right),
\end{equation}

\begin{equation}\label{eq:hbound3}
\lim_{w\to x} \left[h_r(w,\gamma)-e^{kxw}w^{\gamma/2}(w-x)^{-\gamma/2}\zeta(w)^{\gamma/2}\sum_{j=0}^{r}\frac{1}{\Gamma(\frac{\gamma}{2}-j)}\zeta(w)^{-j-1}\right]=\mathcal{O}\left(N^{\frac{\mathrm{Re}(\gamma)}{2}-1}\right),
\end{equation}

\begin{equation}\label{eq:hbound4}
\partial_\gamma \lim_{w\to x} \left[h_r(w,\gamma)-e^{kxw}w^{\gamma/2}(w-x)^{-\gamma/2}\zeta(w)^{\gamma/2}\sum_{j=0}^{r}\frac{1}{\Gamma(\frac{\gamma}{2}-j)}\zeta(w)^{-j-1}\right]=\mathcal{O}\left(\log N N^{\frac{\mathrm{Re}(\gamma)}{2}-1}\right),
\end{equation}

\noindent where the implied constants in the errors are uniform in $\gamma$ in compact subsets of $\lbrace \gamma\in \C:\mathrm{Re}(\gamma)>-2\rbrace$ as well as uniform in $x$ in compact subsets of $(0,1)$.

\end{lemma}

\begin{proof}

Consider first the series expansion of $\zeta(w)$ around $w=x$:

\begin{align}\label{eq:zetaseries}
\zeta(w)&=-(N+k)\left(x(w-x)-\log\left(1+\frac{w-x}{x}\right)\right)\\
\notag &=-(N+k)\left(x(w-x)+\sum_{j=1}^\infty \frac{(-1)^j}{j}\left(\frac{w-x}{x}\right)^j\right)\\
\notag &=(N+k)(w-x)\frac{1-x^2}{x}\left(1+\sum_{j=1}^\infty \frac{(-1)^j}{j+1}\frac{1}{1-x^2}\left(\frac{w-x}{x}\right)^j\right).
\end{align}

\noindent From this, we note that the Taylor coefficients (when expanding around $w=x$) of 

\begin{align*}
w\mapsto e^{kxw}w^{\gamma/2}(w-x)^{-\gamma/2}\zeta(w)^{\gamma/2}
\end{align*}

\noindent can be written explicitly (e.g. in terms of Bell polynomials) and they are of the form

$$
(1-x^2)^{\gamma/2} e^{kx^2} (N+k)^{\gamma/2}c(\gamma,x)
$$

\noindent where $c$ is independent of $N$ and for each $x$, $c(\gamma,x)$ is a polynomial in $\gamma$ (this is just from the fact that the Taylor coefficients of $x\mapsto (1+x)^\gamma$ are generalized binomial coefficients -- polynomials in $\gamma$) and for each $\gamma$, $c(\gamma,x)$ is a rational function in $x$ with possible poles at $x=0$ or $x=\pm 1$). With similar reasoning, the Laurent coefficients of $\zeta(w)^{-j-1}$ are of the form

$$
(N+k)^{-j-1}\rho(x)
$$

\noindent where $\rho$ is a rational function independent of $N$ and $\gamma$ and its possible poles are at $x=0$ and $x=\pm 1$. So combining these two representations, we see that the Laurent coefficients of $e^{kxw}w^{\gamma/2}(w-x)^{-\gamma/2}\zeta(w)^{\gamma/2}\sum_{j=0}^{r}\frac{1}{\Gamma(\frac{\gamma}{2}-j)}\zeta(w)^{-j-1}$ (and in particular $h_{j,r}$ which are just the negative Laurent coefficients) can be written in the form

$$
(1-x^2)^{\gamma/2}e^{kx^2}\sum_{j=0}^r (N+k)^{\frac{\gamma}{2}-j-1}c_j(\gamma,x)\frac{1}{\Gamma(\frac{\gamma}{2}-j)},
$$

\noindent where all we need to know about the functions $c_j(\gamma,x)$ is that they are independent of $N$ (though they do depend on $k$), polynomials in $\gamma$ and rational functions in $x$ with the only possible poles being at $x=0$ or $x=\pm 1$. Note that in our notation, we hide the fact that the function $c_j$ will depend on which Laurent coefficient we are looking at.

Now $h_r(0,\gamma)=\sum_{j=0}^rh_{j,r}(-x)^{-j-1}$ and 

$$
\lim_{w\to x} \left[h_r(w,\gamma)-e^{kxw}w^{\gamma/2}(w-x)^{-\gamma/2}\zeta(w)^{\gamma/2}\sum_{j=0}^{r}\frac{1}{\Gamma(\frac{\gamma}{2}-j)}\zeta(w)^{-j-1}\right],
$$

\noindent which is simply minus the order zero Laurent coefficient of the function $w\mapsto e^{kxw}w^{\gamma/2}(w-x)^{-\gamma/2}\zeta(w)^{\gamma/2}\sum_{j=0}^{r}\frac{1}{\Gamma(\frac{\gamma}{2}-j)}\zeta(w)^{-j-1}$, can both be written in the form $(1-x^2)^{\gamma/2}e^{kx^2}\sum_{j=0}^r (N+k)^{\frac{\gamma}{2}-j-1}c_j(\gamma,x)\frac{1}{\Gamma(\frac{\gamma}{2}-j)}$ (now with a different $c_j$ as before and different in both cases) where $c_j$ is again independent of $N$, polynomial in $\gamma$, rational in $x$, and its only possible poles are at $x=0$ or $x=\pm 1$. From this representation, the analyticity claim along with all the different claims about the bounds are immediate -- we omit further details.
\end{proof}

We now turn to the local parametrix.

\subsection{The local parametrix} Here we look for a function which has the same jump conditions as $S$ in a small enough neighborhood of $x$ and (in the notation of Definition \ref{def:h}) up to a term of order $\mathcal{O}(N^{\frac{\mathrm{Re}(\gamma)}{2}-r-2})$, agrees with $\widehat{P}^{(\infty,r)}$ on the boundary of this neighborhood.\footnote{Note that typically one considers matching conditions up to a term of order $N^{-1}$, but as in our differential identity, there are essentially terms proportional to $N^{\frac{\gamma}{2}+1}$, we need our error terms to be of order $N^{-\frac{\gamma}{2}-2}$. Moreover, as we vary $\gamma$, in our differential identity, we have added this extra parameter $r$ to ensure that throughout the values of $\gamma$ we integrate over, the error stays small.} To do this, let $U$ be a small but fixed circular neighborhood of $x$. We assume that the neighborhood is small enough so that 
$0,1\notin U$. We will also think of $\zeta$ (from \eqref{eq:zeta}) as a coordinate change of this neighborhood -- for this reason, we'll also want $U$ to be small enough that $\zeta$ is one-to-one on it. $\zeta$ blows up $U$ conformally into a large neighborhood of the origin. From the definition of $\Sigma$, it follows that $\zeta$ maps $U\cap \Sigma$ into a segment of the imaginary axis.  We then define our local parametrix in the following way. For $w\in U$, let

\begin{equation}\label{eq:local}
P^{(x,r)}(w)=\begin{pmatrix}
1 &  Q_r(w)\\
0 & 1
\end{pmatrix}\widehat{P}^{(\infty,r)}(w),
\end{equation}

\noindent where

\begin{equation}\label{eq:Q}
Q_r(w)=w^{\frac{\gamma}{2}}(w-x)^{-\frac{\gamma}{2}}\zeta(w)^{\frac{
\gamma}{2}}e^{kxw}\left[\zeta(w)^{-\frac{\gamma}{2}}e^{\zeta(w)}\frac{\Gamma\left(\frac{\gamma}{2},\zeta(w)\right)}{\Gamma(\frac{\gamma}{2})}-\sum_{j=0}^{r}\frac{1}{\Gamma(\frac{\gamma}{2}-j)}\zeta(w)^{-j-1}\right]
\end{equation}

\noindent and $\Gamma(\nu,\zeta)$ is the upper incomplete gamma-function:

\begin{align}\label{eq:incomp}
\Gamma(\nu,\zeta)=\Gamma(\nu)\left(1-\zeta^\nu \gamma^*(\nu,\zeta)\right),
\end{align}

\noindent where $\gamma^*(\nu,\zeta)=e^{-\zeta}\sum_{j=0}^\infty \frac{\zeta^j}{\Gamma(j+\nu+1)}$ is an entire function of $\zeta$, and the branch of the root is the principal one.

\begin{remark}\label{rem:local}
We'll try to clarify the definition of $P^{(x,r)}$ and $\widehat{P}^{(\infty,r)}$ now. Note that in the definition of $P^{(x,r)}$, by our discussion in the previous section, 

$$
w\mapsto w^{\frac{\gamma}{2}}(w-x)^{-\frac{\gamma}{2}}\zeta(w)^{\frac{
\gamma}{2}}e^{kxw}\sum_{j=0}^{r}\frac{1}{\Gamma(\frac{\gamma}{2}-j)}\zeta(w)^{-j-1}
$$

\noindent is a sum of an analytic function and a Laurent polynomial in $(w-x)$ in $U$ $($if we choose it small enough$)$. Thus it does not affect the jump structure of $P^{(x,r)}$. The role of the incomplete gamma function is to produce the desired jump structure. The $\sum_{j=0}^{r}\frac{1}{\Gamma(\frac{\gamma}{2}-j)}\zeta^{-j-1}$-term is required for the matching condition to hold. Indeed $($see \eqref{eq:gammaasy}$)$, this is the beginning of the asymptotic expansion of $\zeta^{-\frac{\gamma}{2}}e^{\zeta}\frac{\Gamma\left(\frac{\gamma}{2},\zeta\right)}{\Gamma(\frac{\gamma}{2})}$ and is valid for large $|\zeta|$. This yields an error of size $N^{\frac{\mathrm{Re}(\gamma)}{2}-r-2}$ in the matching condition. But in addition to having the correct jump and matching conditions, we also need $P^{(x,r)}$ to have the correct type of singularity at $x$ to end up with a small norm problem. For this, we need to counter the singularities at $w=x$ coming from the sum $\sum_{j=0}^{r}\frac{1}{\Gamma(\frac{\gamma}{2}-j)}\zeta(w)^{-j-1}$. This is done by the function $h_r$ and is where the condition \eqref{eq:hcond} comes from.
\end{remark}

To construct a small norm problem, we'll need to know what kind of Riemann-Hilbert problem $P^{(x,r)}$ satisfies. 

\begin{lemma}\label{le:localRHP}
$P^{(x,r)}$ satisfies the following Riemann-Hilbert problem. 

\begin{itemize}[leftmargin=0.5cm]
\item $P^{(x,r)}:U\setminus ([0,x]\cup \Sigma)\to \C^{2\times 2}$ is analytic.
\item $P^{(x,r)}$ has continuous boundary values on $U\cap ([0,x]\cup \Sigma)\setminus \lbrace x\rbrace$ and they satisfy the following jump conditions:

\begin{equation}\label{eq:localjump1}
P^{(x,r)}_+(w)=P^{(x,r)}_-(w)\begin{pmatrix}
0 & (w-x)^{\frac{\gamma}{2}}w^{-\frac{\gamma}{2}}e^{kxw}\\
-(w-x)^{-\frac{\gamma}{2}}w^{\frac{\gamma}{2}}e^{-kxw} & 0
\end{pmatrix}, \qquad w\in \Sigma\setminus \lbrace x\rbrace
\end{equation}

\noindent and 

\begin{equation}\label{eq:localjump2}
P^{(x,r)}_+(w)=P^{(x,r)}_-(w)\begin{pmatrix}
1 & 0\\
2i \sin \frac{\pi\gamma}{2}|w|^{\gamma/2}|w-x|^{-\frac{\gamma}{2}}e^{-kxw}e^{-(N+k)(xw+\ell-\log w)} & 1
\end{pmatrix}, \qquad w\in (0,x).
\end{equation}

\item As $w\to x$ from $\mathrm{Int}(\Sigma)\setminus[0,x]$, 

\begin{equation}\label{eq:localasy}
S(w)P^{(x,r)}(w)^{-1}=\mathcal{O}(1)+\mathcal{O}(|w-x|^{\frac{\mathrm{Re}(\gamma)}{2}})
\end{equation}

\noindent and as $w\to x$ from $\mathrm{Ext}(\Sigma)$, 

\begin{equation}\label{eq:localasyb}
P^{(x,r)}(w)=\mathcal{O}(|w-x|^{-\frac{|\mathrm{Re}(\gamma)|}{2}})
\end{equation}

\noindent where the notation means that each entry satisfies the claimed bound.

\item We have for each fixed $k\in \Z$

\begin{align}\label{eq:localmatch}
P^{(x,r)}(w)\left[\widehat{P}^{(\infty,r)}(w)\right]^{-1}=I+\mathcal{O}(N^{\frac{\mathrm{Re}(\gamma)}{2}-r-2}),
\end{align}

\noindent uniformly in $w\in \partial U$, $\gamma$ in compact subsets of $\lbrace \gamma\in \C:\mathrm{Re}(\gamma)>-2\rbrace$, and $x$ in compact subsets of $(0,1)$.
\end{itemize}
\end{lemma}

\noindent Again for the proof, see Appendix \ref{app:RHPasy}.

We are now in a position to perform our final transformation and complete our asymptotic analysis of $Y_{N+k}$.

\subsection{The final transformation and asymptotic analysis} 

Our final transformation of the problem is the following one (we drop the $r$ dependence from $R$):

\begin{equation}\label{eq:Rdef}
R(w)=\begin{cases}
S(w)\left[P^{(x,r)}(w)\right]^{-1}, & w\in U\\
S(w)\left[\widehat{P}^{(\infty,r)}(w)\right]^{-1}, & w\in \C\setminus \overline{U}
\end{cases}.
\end{equation}

We now describe the RHP $R$ solves (still assuming that $Y$ and hence $R$ exists).

\begin{lemma}\label{le:RRHP}
$R$ is the unique solution to the following RHP$:$

\begin{itemize}[leftmargin=0.5cm]
\item $R:\C\setminus(\partial U\cup \lbrace |w|=1\rbrace)\to \C^{2\times 2}$ is analytic. 
\item $R$ has continuous boundary values on $\partial U\cup \lbrace |w|=1\rbrace$ and these satisfy 

\begin{equation}\label{eq:Rjump1}
R_+(w)=R_-(w)P^{(x,r)}(w)\left[\widehat{P}^{(\infty,r)}(w)\right]^{-1}, \qquad w\in \partial U
\end{equation}

\noindent and 

\begin{equation}\label{eq:Rjump2}
R_+(w)=R_-(w)\widehat{P}^{(\infty,r)}(w)\begin{pmatrix}
1 & 0\\
w^{\frac{\gamma}{2}}(w-x)^{-\frac{\gamma}{2}}e^{-kxw}e^{(N+k)(xw+\ell-\log w)} & 1
\end{pmatrix}\left[\widehat{P}^{(\infty,r)}(w)\right]^{-1},  \ |w|=1.
\end{equation}

\item As $w\to \infty$, $R(w)=I+\mathcal{O}(w^{-1})$.
\end{itemize}

\noindent Moreover, if we write $I+\Delta_R$ for the jump matrix of $R$, then for each fixed $k\in \Z$, as $N\to\infty$, $\sup_{w\in \partial U}|\Delta_R(w)|=\mathcal{O}({N^{\frac{\mathrm{Re}(\gamma)}{2}-r-2}})$ and $\sup_{|w|=1}|\Delta_R(w)|=\mathcal{O}(e^{-cN})$ for some $c>0$.  The implied constants in these estimates are uniform in $\gamma$ in compact subsets of $\lbrace \gamma\in \C:\mathrm{Re}(\gamma)>-2\rbrace$ and for $x$ in a compact subset of $(0,1)$.
\end{lemma}

\noindent Again, the proof is in Appendix \ref{app:RHPasy}.

Now this RHP is one that's normalized at infinity and whose jump matrix is close to the identity when $N\to\infty$. Thus it can be solved asymptotically through the standard machinery. In particular if we take $N$ large enough (possibly depending on $\gamma$), then a solution exists. Then reversing the transformations, this implies that $Y$ exists for large enough $N$. In addition to existence, the standard machinery yields the following estimate. 

\begin{lemma}\label{le:smnorm}
Let $K$ be a compact subset of $\lbrace \gamma\in \C:\mathrm{Re}(\gamma)>-2\rbrace$ and let $r>\sup_{\gamma\in K}\mathrm{Re}(\gamma/2)-2$. Then there exists a $N_0=N_0(K)$ such that for $N\geq N_0$, a unique solution to the RHP of Lemma \ref{le:RRHP} exists. Let $\Gamma_R = \partial U \cup \lbrace |w| = 1\rbrace$ be the jump contour of $R$. Then as $N\to\infty$, 
\begin{align}\label{eq:R1}
R(w) = I + \mathcal{O}(N^{\frac{\mathrm{Re}(\gamma)}{2}-r-2}),
\qquad \lim_{w \to \infty} w[R(w) - I] = \mathcal{O}(N^{\frac{\mathrm{Re}(\gamma)}{2}-r-2})
\end{align}

\noindent uniformly in $w \in \mathbb{C} \setminus \Gamma_R$, $\gamma\in K$ and for $x$ in a compact subset of $(0,1)$. Moreover, for any fixed $\epsilon>0$
\begin{align}\label{eq:R2}
\partial_\gamma R(w) = \mathcal{O}(N^{\frac{\mathrm{Re}(\gamma)}{2}-r-2+\epsilon})
\qquad \partial_\gamma \left[\lim_{w \to \infty} w(R(w)-I)\right] = \mathcal{O}(N^{\frac{\mathrm{Re}(\gamma)}{2}-r-2+\epsilon}),
\end{align}

\noindent uniformly in $w\in \C\setminus \Gamma_R$, $\gamma\in K$, as well as uniformly in $x$ when restricted to a compact subset of $(0,1)$.
\end{lemma}

\begin{remark}\label{rem:zones}
Note that when integrating our differential identity, we can choose $K$ to be the integration and by choosing $r$ large enough, our error term will be uniformly small throughout the integration contour -- ensuring that we can essentially ignore $R$ when evaluating our differential identity.
\end{remark}

Armed with these estimates, we now turn to studying the asymptotic behavior of the differential identity and proving Theorem \ref{th:main}.

\section{Proof of Theorem \ref{th:main}: integrating the differential identity}\label{sec:diint}
We summarize the asymptotics of our differential identity in the following lemma.

\begin{lemma}\label{le:diasymp}
Let $\mathrm{Re}(\gamma)>-2$ and let $\gamma$ be such that $D_j^{(N)}(F;\gamma)\neq 0$ for all $j\leq N+1$. Then as $N\to\infty$, 

\begin{align}\label{eq:diasymp}
\partial_{\gamma} \log D_{N-1}(F; \gamma)
& = \frac{Nx^2}{2} + \partial_\gamma \sum_{j=0}^{N-1} \log \frac{\Gamma(\frac{\gamma}{2}+j+1)}{N^{\frac{\gamma}{2}}} + o(1).
\end{align}

\noindent Moreover, if $K$ is a compact subset of $\lbrace \gamma\in \C:\mathrm{Re}(\gamma)>-2\rbrace$, then the $o(1)$ error is uniform in  $\lbrace\gamma\in K: D_j^{(N)}(F; \gamma) \ne 0  \ for \ all \ j\leq N+1\rbrace$, and $x$ in compact subsets of $(0,1)$.
\end{lemma}

\begin{proof}
Let $\gamma$ be such that $\Re(\gamma) > -2$ and $D_j^{(N)}(F; \gamma) \ne 0$ for all $j\leq N+1$. If $\Re(\gamma) > 0$ choose a non-negative integer $r$ such that $\Re(\gamma) - r \le \frac{1}{2}$, otherwise set $r = 0$. Such a choice of $r$ satisfies the following inequality:
\begin{align*}
\max\left( \Re(\gamma) - r - 2, \frac{\Re(\gamma)}{2} - r - 2\right) \le -\frac{3}{2}.
\end{align*}

Fix $\epsilon > 0$ small. We start with the terms that require the evaluation of $Y(w)$ at $w = 0 \in \mathrm{Int}(\Sigma)$. In particular, we will first consider the logarithmic derivatives of $\chi$ and $\widehat{\chi}$. We begin by noting that $g(0) = \ell$ and the global parametrix is given by (see \eqref{eq:hbound1})
\begin{align*}
\widehat{P}^{(\infty, r)}_{N+k}(0) = \begin{pmatrix}
-h_r(0; \gamma)  & 1 \\ -1 & 0 \end{pmatrix}
= \begin{pmatrix}
\mathcal{O}(N^{\frac{\Re(\gamma)}{2}-1})  & 1 \\ -1 & 0 \end{pmatrix}.
\end{align*}

\noindent Let us look at the leading coefficients of our orthogonal polynomials: for each $k \in \mathbb{N}$ we have
\begin{align*}
\chi_{N+k}\widehat{\chi}_{N+k} &= - Y_{N+k+1, 21}(0),
\qquad \widehat{\chi}_{N+k} = \chi_{N+k} \frac{\pi \Gamma \left(1+\frac{\gamma}{2}+N+k\right)}{N^{1+\frac{\gamma}{2} + N + k}}.
\end{align*}

\noindent With the error control \eqref{eq:R1} for the $R$, we have
\begin{align*}
Y_{N+k, 21}(0) 
= T_{N+k, 21}(0) 
= \left[R_{N+k}(0)\widehat{P}_{N+k}^{(\infty, r)}(0)\right]_{21}
&= -1 + \mathcal{O}\left(N^{-\frac{3}{2}}\right).
\end{align*}

\noindent Therefore

\begin{align}\label{eq:chiasy1}
\chi_{N+k} = \left(\frac{\pi \Gamma \left(1+ \frac{\gamma}{2}+N + k\right)}{N^{1+\frac{\gamma}{2}+N + k}}\right)^{-\frac{1}{2}} \left(1+  \mathcal{O}\left(N^{-\frac{3}{2}}\right)\right)
\end{align}
\noindent and 

\begin{align}\label{eq:chiasy2}
\widehat{\chi}_{N+k} = \left(\frac{\pi \Gamma \left(1+ \frac{\gamma}{2}+N + k\right)}{N^{1+\frac{\gamma}{2}+N + k}}\right)^{\frac{1}{2}}  \left(1+  \mathcal{O}\left(N^{-\frac{3}{2}}\right)\right).
\end{align}

\noindent For the logarithmic derivatives, we see from \eqref{eq:hbound2} and \eqref{eq:R2} that for any $\epsilon>0$, 
\begin{align*}
\partial_{\gamma} Y_{N+k, 21}(0)
& = \partial_{\gamma} \left[R_{N+k}(0)\widehat{P}_{N+k}^{(\infty, r)}(0)\right]_{21}
= \mathcal{O}\left(N^{-\frac{3}{2}+\epsilon}\right).
\end{align*}

\noindent Recalling the standard asymptotics of the digamma function (which follow from Binet's second formula for the log-Gamma function, see e.g. \cite[Section 12.32]{Whittaker})
\begin{align*}
\frac{\Gamma'(u)}{\Gamma(u)} = \log u - \frac{1}{2u} + \mathcal{O}(u^{-2}), \qquad u \to \infty,
\end{align*}

\noindent we obtain
\begin{align}\label{eq:dchi1}
\frac{\partial_\gamma \chi_N}{\chi_N} &= -\frac{1}{2} \partial_\gamma \log \frac{\Gamma\left(1 + \frac{\gamma}{2}+N\right)}{N^{\frac{\gamma}{2}}} + \frac{1}{2} \frac{\partial_\gamma Y_{N+1, 21}(0)}{Y_{N+1, 21}(0)}\\
\notag &= -\frac{\gamma + 1}{8N} +\mathcal{O}\left(N^{-\frac{3}{2}+\epsilon}\right)
\end{align}
\noindent and

\begin{align}\label{eq:dchi2}
\frac{\partial_\gamma \widehat{\chi}_N}{\widehat{\chi}_N} &= \frac{\gamma + 1}{8N} + \mathcal{O}\left(N^{-\frac{3}{2}+\epsilon}\right).
\end{align}

\noindent In particular,
\begin{align}\label{eq:dia1}
-\left(N + \frac{\gamma}{2}\right) \frac{\partial_\gamma \widehat{\chi}_N}{\widehat{\chi}_N} -N \frac{\partial_\gamma \chi_N}{\chi_N}
& = \mathcal{O}\left(N^{-\frac{1}{2} + \epsilon}\right).
\end{align}

\noindent  Note that these estimates are all uniform in  compact subsets of $\{\gamma \in \mathbb{C}: \Re (\gamma) > -2\}$ (as long as the relevant polynomials exist) and if we choose $\epsilon$ small enough, $\mathcal{O}\left(N^{-\frac{1}{2}+\epsilon}\right)=o(1)$ uniformly in everything relevant.

We now consider the $p_{N+1}(0)\partial_\gamma q_N(0)$-term. Using \eqref{eq:hbound1} we first get
\begin{align}\label{eq:pN10}
\frac{p_{N+1}(0)}{\chi_{N+1}}
&= Y_{N+1,11}(0)\\
\notag & = e^{(N+1)\ell} \left[R_{N+1}(0) \widehat{P}_{N+1}^{(\infty, r)}(0)\right]_{11}\\
\notag & = - e^{(N+1) \ell} \left[h_r(0; \gamma) \left(1 + \mathcal{O}\left(N^{\frac{\Re(\gamma)}{2} - r - 2}\right)\right) + \mathcal{O}\left(N^{\frac{\Re(\gamma )}{2} - r - 2}\right)\right]\\
\notag & = e^{(N+1)\ell}\left[\mathcal{O}(N^{\frac{\Re(\gamma)}{2}-1}) + \mathcal{O}(N^{-\frac{3}{2}})\right].
\end{align}

\noindent Next, we need to evaluate $Y$ at $\infty$, which requires the global parametrix for $w \in \mathrm{Ext}(\Sigma)$:
\begin{align*}
\widehat{P}_{N+k}^{(\infty, r)}(w) = 
\begin{pmatrix}
w^{\frac{\gamma}{2}} (w-x)^{-\frac{\gamma}{2}}
& w^{-\frac{\gamma}{2}} (w-x)^{\frac{\gamma}{2}} h_r(w; \gamma)\\
0 & w^{-\frac{\gamma}{2}} (w-x)^{\frac{\gamma}{2}}
\end{pmatrix}.
\end{align*}

\noindent Using the asymptotics of $R$, namely \eqref{eq:R1}, one has
\begin{align*}
q_N(0) 
&= -\frac{1}{\chi_N} \lim_{w \to \infty} w^{-N}Y_{N+1, 21}(w)\\
& = -\frac{1}{\chi_N} e^{-(N+1)\ell} \lim_{w \to \infty} w\left[R_{N+1}(w)\widehat{P}_{N+1}^{(\infty, r)}(w)\right]_{21}\\
& = -\frac{1}{\chi_N} e^{-(N+1)\ell} \lim_{w\to\infty} wR_{N+1, 21}(w)\\
&=  -\frac{1}{\chi_N} e^{-(N+1)\ell}  \mathcal{O}\left(N^{\frac{\Re(\gamma)}{2} - r - 2}\right).
\end{align*}

\noindent Similarly for the derivative term we find from \eqref{eq:R2}
\begin{align*}
\partial_{\gamma} q_N(0)
&= -\frac{\partial_\gamma \chi_N}{\chi_N}q_N(0) - \frac{1}{\chi_N} e^{-(N+1)\ell} \partial_\gamma  \lim_{w \to \infty} wR_{N+1, 21}(w)\\
& = -\frac{1}{\chi_N} e^{-(N+1)\ell}  \mathcal{O}\left(N^{\frac{\Re(\gamma)}{2} - r - 2+\epsilon}\right).
\end{align*}

\noindent Finally combining this with \eqref{eq:pN10}, \eqref{eq:chiasy1}, and \eqref{eq:chiasy2}  yields the asymptotics of the relevant term:
\begin{align} \label{eq:dia2}
&  -Nx\frac{p_{N+1}(0)}{{\chi}_{N+1}}\frac{\partial_\gamma q_N(0)}{\widehat{\chi}_N} 
= \mathcal{O}\left(N^{-\frac{3}{2}+\epsilon}\right),
\end{align}

\noindent which again under our assumptions is $o(1)$ uniformly in everything relevant.

We now move onto the $\kappa$-terms: we find by the definition of $\kappa$ and $Y$ (along with \eqref{eq:R1}) that
\begin{align*}
\frac{\kappa_{N+k}}{\chi_{N+k}} 
&= \lim_{w \to \infty} w^{-N-k+1} (Y_{N+k, 11}(w) - w^{N+k})\\
& = \lim_{w \to \infty} w\left(T_{N+k, 11}(w) - 1\right)\\
&= \lim_{w \to \infty} w\left(\left[R_{N+k}(w)\widehat{P}_{N+k}^{(\infty, r)}(w)\right]_{11}  - 1\right)\\
&= \frac{\gamma}{2}x + \mathcal{O}\left(N^{\frac{\Re(\gamma)}{2} - r - 2}\right).
\end{align*}

\noindent Similarly from \eqref{eq:R2}, we see that 
\begin{align*}
\frac{\partial_\gamma \kappa_{N+k}}{\chi_{N+k}}
& =  \partial_{\gamma} \frac{\kappa_{N+k}}{\chi_{N+k}} + \frac{\kappa_{N+k}}{\chi_{N+k}} \frac{\partial_{\gamma} \chi_{N+k}}{\chi_{N+k}}\\
& = \partial_\gamma \lim_{w \to \infty} w \left(\left[ R_{N+k}(w) \widehat{P}_{N+k}^{(\infty, r)}(w)\right]_{11}-1\right) + \frac{\kappa_{N+k}}{\chi_{N+k}} \frac{\partial_{\gamma} \chi_{N+k}}{\chi_{N+k}}\\
& = \frac{1}{2}x + \frac{\kappa_{N+k}}{\chi_{N+k}} \frac{\partial_{\gamma} \chi_{N+k}}{\chi_{N+k}} +  \mathcal{O}\left(N^{\frac{\Re(\gamma)}{2} - r - 2 +\epsilon}\right).
\end{align*}

\noindent Therefore, we have
\begin{align} \label{eq:dia3}
Nx \left( \frac{\partial_\gamma \kappa_N}{\chi_N} - \frac{\partial_\gamma \chi_N}{\chi_N} \frac{\kappa_{N+1}}{\chi_{N+1}}\right)
& =Nx\left(\frac{x}{2}+\frac{\partial_\gamma\chi_N}{\chi_N}\left(\frac{\kappa_N}{\chi_N}-\frac{\kappa_{N+1}}{\chi_{N+1}}\right)+\mathcal{O}\left(N^{\frac{\Re(\gamma)}{2} - r - 2 +\epsilon}\right)\right)\\
\notag &=\frac{Nx^2}{2} +  \mathcal{O}\left(N^{-\frac{1}{2} +\epsilon}\right).
\end{align}

The remaining terms in the differential identity (those involving the Cauchy-transforms) require $Y$ near the singularity $w=x$ and hence the local parametrix. For $w \in \mathrm{Int}(\Sigma) \setminus [0,x]$,
\begin{align*}
 Y_{N+k}(w) & = e^{(N+k)\frac{\ell}{2}\sigma_3} R_{N+k}(w) P_{N+k}^{(x, r)}(w) \begin{pmatrix}
1 & 0 \\
w^{\frac{\gamma}{2}}(w-x)^{-\frac{\gamma}{2}} e^{-kxw} e^{-(N+k)(xw + \ell - \log w )} & 1 \end{pmatrix} \\
&\quad \times  e^{(N+k)xw\sigma_3} e^{(N+k)\frac{\ell}{2} \sigma_3}.
\end{align*}

\noindent A straightforward computation shows that
\begin{align*}
&P_{N+k}^{(x, r)}(w) \begin{pmatrix}
1 & 0 \\
w^{\frac{\gamma}{2}}(w-x)^{-\frac{\gamma}{2}} e^{-kxw} e^{-(N+k)(xw + \ell - \log w )} & 1 \end{pmatrix}
= \begin{pmatrix}
\widetilde{P}_{N+k}(w) & e^{kxw} \\
-e^{-kxw} & 0 
\end{pmatrix}
\end{align*}

\noindent where
\begin{align*}
\widetilde{P}_{N+k}(w) 
& = - e^{-kxw}[h_r(w; \gamma) + Q_r(w)] + w^{\frac{\gamma}{2}}(w-x)^{-\frac{\gamma}{2}} e^{\zeta(w)}\\
& = w^{\frac{\gamma}{2}}(w-x)^{-\frac{\gamma}{2}} e^{\zeta(w)} \left(1 - \frac{\Gamma(\frac{\gamma}{2}, \zeta(w))}{\Gamma(\frac{\gamma}{2})}\right) \\
& \qquad -e^{-kxw} \left[ h_r(w; \gamma) - e^{kxw} w^{\frac{\gamma}{2}}(w-x)^{-\frac{\gamma}{2}} \zeta(w)^{\frac{\gamma}{2}} \sum_{j=0}^{r} \frac{1}{\Gamma(\frac{\gamma}{2} - j)} \zeta(w)^{-j-1}\right]\\
& \xrightarrow{w \to x} \frac{(N+k)^{\frac{\gamma}{2}}(1-x^2)^{\frac{\gamma}{2}}}{\Gamma(1+\frac{\gamma}{2})} + \mathcal{O}\left(N^{\frac{\Re(\gamma)}{2}-1}\right)\\
& = \mathcal{O}\left(N^{\frac{\Re(\gamma)}{2}}\right)
\end{align*}

\noindent where we used \eqref{eq:hbound3} and \eqref{eq:zetaseries}. Taking the limit $w \to x$ inside the set $\mathrm{Int}(\Sigma) \setminus [0,x]$, we get (from the above bound on $\widetilde{P}$ as well as \eqref{eq:R1}):
\begin{align*}
& \lim_{w \to x} Y_{N+k}(w)\\
& \qquad =\begin{pmatrix}
x^{N+k} \left[R_{N+k, 11}(x) \lim_{w \to x} \widetilde{P}_{N+k}(w) - R_{N+k, 12}(x) e^{-kx^2}\right]&
e^{-Nx^2} R_{N+k, 11}(x)\\
e^{(N+k)x^2}\left[R_{N+k, 21}(x) \lim_{w \to x} \widetilde{P}_{N+k}(x) - R_{N+k, 22}(x) e^{-kx^2}\right] &
x^{-(N+k)}R_{N+k, 21}(x) e^{kx^2}
\end{pmatrix}\\
& \qquad =\begin{pmatrix}
\mathcal{O}\left(N^{\frac{\Re(\gamma)}{2}}\right)&
e^{-Nx^2}\left(1+\mathcal{O}\left(N^{\frac{\Re(\gamma)}{2} - r - 2}\right)\right)\\
-e^{Nx^2}\left[1 + \mathcal{O}(N^{\frac{\Re(\gamma)}{2} - r - 2}) + \mathcal{O}\left(N^{\Re(\gamma) - r - 2}\right) \right] &
x^{-(N+k)}e^{kx^2}\mathcal{O}\left(N^{\frac{\Re(\gamma)}{2} - r - 2}\right)
\end{pmatrix}.
\end{align*}

\noindent We immediately see that
\begin{align}\label{eq:cauchyp}
\lim_{z \to x} \oint_{\Sigma}p_N(w)w^{-N+1}\frac{f(w)}{w-z}\frac{dw}{2\pi i w}
& = \chi_N \lim_{z \to x} Y_{N, 12}(z)
= \chi_N e^{-Nx^2}\left(1+\mathcal{O}\left(N^{-\frac{3}{2}}\right)\right).
\end{align}

\noindent A similar argument (using \eqref{eq:R2}, using Lemma \ref{le:h} for the asymptotics of $\partial_\gamma \widetilde{P}$, and \eqref{eq:dchi1}) shows that 
\begin{align}\label{eq:dq}
x^N \frac{\partial q_N(x^{-1})}{\partial\gamma}
= -\frac{\partial}{\partial {\gamma}} \left[\frac{1}{\chi_N} Y_{N+1, 21}(x)\right]
& = -\frac{1}{\chi_N} \partial_{\gamma}Y_{N+1, 21}(x) + \frac{\partial_{\gamma}\chi_N}{\chi_N} \frac{1}{\chi_N} Y_{N+1, 21}(x)\\
&\notag = \frac{e^{Nx^2}}{\chi_N}\mathcal{O}(N^{-\frac{3}{2} + \epsilon}).
\end{align}

\noindent Combining with \eqref{eq:cauchyp}, we find (again with the required uniformity)
\begin{align} \label{eq:dia4}
\frac{\gamma}{2}x^N \frac{\partial q_N(x^{-1})}{\partial \gamma}\lim_{z \to x} \oint_{\Sigma}p_N(w)w^{-N+1}\frac{f(w)}{w-z}\frac{dw}{2\pi i w}
& = \mathcal{O}(N^{-\frac{3}{2} + \epsilon}).
\end{align}

Similarly, we have (again for $z\in\mathrm{Int}(\Sigma)\setminus[0,x]$)
\begin{gather*}
\lim_{z \to x} \oint_\Sigma \frac{q_N(w^{-1})f(w)}{w-z}\frac{dw}{2\pi i w}
= -\frac{1}{\chi_N} \lim_{z \to x} Y_{N+1, 22}(z)
= -\frac{1}{\chi_N} x^{-(N+1)} \mathcal{O}(N^{\frac{\Re(\gamma)}{2} - r - 2}), \\
\frac{\partial p_N(x)}{\partial \gamma}= \frac{\partial}{\partial \gamma}\left[\chi_N Y_{N, 11}(x)\right]= \chi_N \partial_\gamma Y_{N, 11}(x) + \frac{\partial \chi_N}{\chi_N}\chi_N Y_{N, 11}(x)= \chi_N x^{N+1}\mathcal{O}(N^{\frac{\Re(\gamma)}{2}+\epsilon}),
\end{gather*}

\noindent which implies
\begin{align} \label{eq:dia5}
-\frac{\gamma}{2}\frac{\partial p_N(x)}{\partial \gamma}\oint_\Sigma \frac{q_N(w^{-1})f(w)}{w-x}\frac{dw}{2\pi i w}
& =\mathcal{O}(N^{\frac{\Re(\gamma)}{2} - r - 2+\epsilon}) = \mathcal{O}(N^{-\frac{3}{2} + \epsilon}).
\end{align}

Finally our lemma follows by substituting \eqref{eq:dia1}, \eqref{eq:dia2}, \eqref{eq:dia3}, \eqref{eq:dia4} and \eqref{eq:dia5} into the differential identity in Lemma \ref{le:DI}. As mentioned, the $o(1)$ error is uniform in compact subsets of $\Re(\gamma) > -2$ if we take $\epsilon$ small enough. The uniformity in $x$ follows from the corresponding uniformity in $x$ in our asymptotic estimates for $R$.

\end{proof}

\begin{proof}[Proof of Theorem \ref{th:main}]
Consider now some $\gamma\in \C$, which may depend on $N$ but is within a fixed compact subset of $\lbrace \gamma\in \C:\mathrm{Re}(\gamma)>-2\rbrace$. We wish to write $\log D_{N-1}(F;\gamma)=\log D_{N-1}(F;0)+ \int_0^\gamma \partial_s D_{N-1}(F;s)ds$ along some suitable integration contour in the complex plane, and use Lemma \ref{le:diasymp} to estimate this integral. The issue being that we need to be able to ensure the condition $D_j^{(N)}(F;\gamma)\neq 0$ for $j\leq N+1$ throughout the whole contour (or say apart from a finite number of points of it). To ensure this, note that from the determinantal representation

$$
D_j^{(N)}(F;\gamma)=\det \left(\int_\C w^k\overline{w}^l |w-x|^\gamma e^{-N|w|^2}d^2w\right)_{k,l=0}^j 
$$

\noindent $\gamma \mapsto D_j^{(N)}(F;\gamma)$ is analytic for each $j$ and from e.g. a variant of Lemma \ref{le:heine}, one can see that this is a non-trivial analytic function. Thus in any compact subset of $\lbrace \gamma\in \C:\mathrm{Re}(\gamma)>-2\rbrace$, $\gamma\mapsto D_{j}^{(N)}(F;\gamma)$ has only finitely many zeroes and in any such compact set, there are only finitely many points $\gamma$ at which even one of the $D_j^{(N)}(F;\gamma)$ (for $j\leq N+1$) vanishes. In particular, for any $\gamma\in \C$ which is within some fixed compact subset of $\lbrace \gamma\in \C:\mathrm{Re}(\gamma)>-2\rbrace$, we have for any smooth simple contour from $0$ to $\gamma$ such that, we have $D_{j}^{(N)}(F;s)\neq 0$  for all $j\leq N+1$ for all but finitely many points on the contour. 

Let us assume further that $\gamma$ is such that $D_{j}^{(N)}(F;\gamma)\neq 0$ for all $j\leq N+1$. Then from Lemma \ref{le:diasymp} we see that when integrating along the straight line from $0$ to $\gamma$, 

\begin{align*}
\log D_{N-1}(F;\gamma)&=\log D_{N-1}(F;0)+\int_0^\gamma \left(\frac{Nx^2}{2}+\partial_s \sum_{j=0}^{N-1}\log \frac{\Gamma(\frac{s}{2}+j+1)}{N^{s/2}} +o(1)\right)ds\\
&=\log D_{N-1}(F;0)+N\gamma\frac{x^2}{2}+\sum_{j=0}^{N-1}\left(\log\frac{\Gamma(\frac{\gamma}{2}+j+1)}{N^{\gamma/2}}-\log \Gamma(j+1)\right)+o(1),
\end{align*}

\noindent where  we have made critical use of the uniformity in Lemma \ref{le:diasymp}. Now given that $G(u+1) = \Gamma(u) G(u)$ and $G(1) = 1$, we see that

\begin{align*}
\sum_{j=0}^{N-1} \log \frac{\Gamma(\frac{\gamma}{2}+j+1)}{N^{\frac{\gamma}{2}}\Gamma(j+1)}
& = \log \frac{G(\frac{\gamma}{2} + N + 1)}{G(1+\frac{\gamma}{2})G(N+1)} - \frac{N\gamma}{2} \log N.
\end{align*}

Let us recall the asymptotics for the logarithm of Barnes G-function (see e.g. \cite[Theorem 1 and Theorem 2]{FL}):
\begin{align*}
\log G(u+1) = \frac{1}{12} - \log A + \frac{u}{2} \log 2\pi + \left(\frac{u^2}{2} - \frac{1}{12}\right)\log u - \frac{3u^2}{4} + O(u^{-2})
\end{align*}

\noindent where $A$ is the Glaisher-Kinkelin constant. In particular

\begin{align*}
\log \frac{G(\frac{\gamma}{2}+N+1)}{G(N+1)}&= \frac{\gamma}{4} \log 2\pi + \left(\frac{(N+\frac{\gamma}{2})^2}{2} - \frac{N^2}{2}\right)\log N\\
& \qquad + \left[\frac{(N+\frac{\gamma}{2})^2}{2}-\frac{1}{12}\right] \log \left(1+\frac{\gamma}{2N}\right)
-\frac{3}{4}\left[(N+\frac{\gamma}{2})^2 - N^2\right]\\
& = \frac{\gamma}{4}\log 2 \pi + \left[\frac{N\gamma}{2} + \frac{\gamma^2}{8}\right] \log N - \frac{N\gamma}{2} + o(1).
\end{align*}

\noindent Therefore
\begin{align*}
\log \frac{D_{N-1}(F; \gamma)}{D_{N-1}(F; 0)}
= \frac{N \gamma}{2}(x^2-1) + \frac{\gamma}{4} \log 2\pi + \frac{\gamma^2}{8} \log N -\log G(1+\frac{\gamma}{2}) + o(1)
\end{align*}

\noindent where the error is still uniform in $\gamma$ in compact subsets $\lbrace \gamma\in \C:\mathrm{Re}(\gamma)>-2\rbrace$ assuming that $D_{j}^{(N)}(F,\gamma)\neq 0$ for all $j\leq N+1$.  Now by this uniformity, if $\gamma_0$ is such that $D_j^{(N)}(F;\gamma_0)=0$ for some $j\leq N+1$, and the fact (that we already noted) that by analyticity, the zeroes of $\gamma \mapsto D_j^{(N)}(F;\gamma)$ are isolated so by continuity of $\gamma\mapsto D_{N-1}(F;\gamma)$, we can simply let $\gamma\to \gamma_0$ in the above formula, and we see that the above formula is true (with uniform error bounds) for all $\gamma$ in a compact subset of $\lbrace \gamma:\mathrm{Re}(\gamma)>-2\rbrace$. 

To conclude, recall that we already argued that by rotation invariance of the law of the eigenvalues, $\E|\det(G_N-z)|^\gamma=\E|\det(G_N-x)|^\gamma$ for $|z|=x$, so by Lemma \ref{le:heine} (applied to the function $F(z)=1$ which corresponds to $\gamma=0$) we arrive at
\begin{align} \label{eq:result}
\mathbb{E}|\det (G_N - z)|^{\gamma}
\notag & = \frac{N!}{Z_N}D_{N-1}(F; 0) \frac{D_{N-1}(F; \gamma)}{D_{N-1}(F; 0)}\\
& =  N^{\frac{\gamma^2}{8}} e^{\frac{N\gamma}{2}(|z|^2-1)} \frac{(2\pi)^{\frac{\gamma}{4}}}{G(1+\frac{\gamma}{2})}(1+o(1))
\end{align}

\noindent with the required uniformity.
\end{proof}

\appendix

\section{Orthogonal polynomials -- Proofs for the results in Section \ref{sec:ginortho}}\label{app:ortho}
In this appendix we prove Lemma \ref{le:ortoc} and Lemma \ref{le:ortoq}. We begin with our proof of Lemma \ref{le:ortoc}, which is essentially that of \cite[the proof of Lemma 3.1]{BBLM}.

\begin{proof}[Proof of Lemma \ref{le:ortoc}]
For $w\in \C\setminus (-\infty,x]$, let 

\begin{equation*}
h(w)=(w-x)^{\frac{\gamma}{2}}\int_x^{\overline{w}}(s-x)^{\frac{\gamma}{2}+k} e^{-Nw s}ds, 
\end{equation*}

\noindent where the roots are according to the principal branch, and the integration contour does not intersect $(-\infty,x)$. One has 

\begin{equation*}
\frac{\partial}{\partial \overline{w}}h(w)=|w-x|^\gamma (\overline{w}-x)^k e^{-N|w|^2},
\end{equation*}

\noindent so we see by \eqref{eq:ortop} (under our assumption of $D_{j-1}^{(N)}(F),D_j^{(N)}(F)\neq 0$ which implies the existence of $p_j$) and Green's theorem\footnote{One can check that in the definition of $h$ (and similarly its derivatives), the jumps along $(-\infty,x)$, coming from the roots, cancel so the partial derivatives of $h$ are continuous apart from possibly at $w=x$. Here the possible singularity of $\partial_{\overline{w}}h$ is still  integrable in the plane (as we assume $\mathrm{Re}(\gamma)>-2$) so one can justify the use of Green's theorem with a simple limiting argument.}  that for $k\leq j$

\begin{align*}
\frac{1}{\chi_j}\delta_{j,k}&=\int_{\C}p_j(w)\overline{w}^k |w-x|^\gamma e^{-N|w|^2}d^2 w\\
&=\int_{\C}p_j(w)(\overline{w}-x)^k |w-x|^\gamma e^{-N|w|^2}d^2 w\\
&=\lim_{r\to\infty}\int_{|w|\leq r}\frac{\partial }{\partial \overline{w}}\left[p_j(w)h(w)\right]d^2 w\\
&=\lim_{r\to\infty}\frac{1}{2i}\oint_{|w|=r}p_j(w)h(w)dw.
\end{align*}

We now wish to deform the $\lbrace |w|=r\rbrace$ contour into $\Sigma$. To do this, we note that for $|w|=r$

\begin{align*}
h(w)=(w-x)^{\frac{\gamma}{2}}\left(\int_x^{\overline{w}\times \infty}(s-x)^{\frac{\gamma}{2}+k}e^{-Nws}ds-\int_{\overline{w}}^{\overline{w}\times \infty}(s-x)^{\frac{\gamma}{2}+k}e^{-Nws}ds\right),
\end{align*}

\noindent where again we take the contours to not intersect $(-\infty,x)$. The second integral is easily seen to be $\mathcal{O}(e^{-\frac{1}{2}|r|^2 N})$ uniformly on $\lbrace |w|=r\rbrace$. For the first integral we note that 

\begin{align*}
(w-x)^{\frac{\gamma}{2}}\int_x^{\overline{w}\times \infty}(s-x)^{\frac{\gamma}{2}+k}e^{-Nws}ds&=N^{-\frac{\gamma}{2}-k-1} w^{-k-1} w^{-\frac{\gamma}{2}}(w-x)^{\frac{\gamma}{2}} e^{-Nxw}\Gamma\left(\frac{\gamma}{2}+k+1\right)\\
&=\frac{\pi \Gamma\left(\frac{\gamma}{2}+k+1\right)}{N^{\frac{\gamma}{2}+k+1}}  w^{-k} \frac{f(w)}{\pi w}
\end{align*}

\noindent which is an analytic function of $w$ in $\C\setminus [0,x]$. We thus see by contour deformation and our bound on the second integral that if $\Sigma$ is a simple closed contour encircling $[0,x]$ (not passing through any point of the interval) 
\begin{align*}
\frac{1}{\chi_j}\delta_{j,k}=\frac{\pi \Gamma\left(\frac{\gamma}{2}+k+1\right)}{N^{\frac{\gamma}{2}+k+1}} \oint_{\Sigma} p_j(w)w^{-k}f(w)\frac{dw}{2\pi i w},
\end{align*}

\noindent which was precisely the claim. The only remaining issue is to consider the case where $\Sigma$ passes through $x$. Let $\epsilon>0$ and let $\Sigma_\epsilon$ be an indentation of $\Sigma$ at $x$ such that $\Sigma_\epsilon$ does not pass through $x$ nor any other point of $[0,x]$. We then have 

\begin{equation*}
\oint_{\Sigma} p_j(w) w^{-k}f(w)\frac{dw}{2\pi i w}=\oint_{\Sigma_\epsilon}p_j(w)w^{-k}f(w)\frac{dw}{2\pi i w}+\oint_{C_\epsilon}p_j(w)w^{-k}f(w)\frac{dw}{2\pi i w},
\end{equation*}

\noindent where $C_\epsilon=(\Sigma\setminus \Sigma_\epsilon)\cup (\Sigma_\epsilon\setminus \Sigma)$ with a suitable orientation. The first integral here is precisely what we want the left hand side to be for each $\epsilon>0$ and since  the possible singularity of $f$ at $x$ is integrable,  we see that as $\epsilon\to 0$, the second integral vanishes. This concludes the proof.
\end{proof}

We now turn to the proof of Lemma \ref{le:ortoq}.

\begin{proof}[Proof of Lemma \ref{le:ortoq}]
We begin by noting that with a simple modification of the argument of the proof of Lemma \ref{le:ortoc}, one finds that 
\begin{align}\label{eq:moms}
\int_\C w^j\overline{w}^k F(w)e^{-N|w|^2}d^2w=\sum_{l=0}^k {{k}\choose{l}}x^{k-l}\frac{\pi \Gamma(1+\frac{\gamma}{2}+l)}{N^{1+\frac{\gamma}{2}+l}}\oint_{\Sigma}w^{-(l-j)}f(w)\frac{dw}{2\pi i w}.
\end{align}

Now if we write $M$ for the moment matrix $M_{jk}=\int_\C w^j\overline{w}^k F(w)e^{-N|w|^2}d^2w$, $M'$ for the moment matrix $M'_{lj}=\oint_{\Sigma}w^{-(l-j)}f(w)\frac{dw}{2\pi i w}$ and $T$ for the upper triangular matrix $T_{lk}={{k}\choose{l}}x^{k-l}\frac{\pi \Gamma(1+\frac{\gamma}{2}+l)}{N^{1+\frac{\gamma}{2}+l}}\mathbf{1}\lbrace l\leq k\rbrace$, then \eqref{eq:moms} can be written as $M=TM'$. Taking the determinant of this identity and using the fact that $\det T$ is the product of the diagonal elements of $T$ yields \eqref{eq:Dhat}.

To prove \eqref{eq:qporto}, let us first notice that from \eqref{eq:Dhat}, $q_j$ exists under our assumptions. For $k<j$, we note that \eqref{eq:qporto} again follows from noting that the linearity of the determinant implies that the determinantal representation of $\oint_\Sigma w^k q_j(w^{-1})f(w)\frac{dw}{2\pi i w}$ has two identical columns and thus vanishes. For $k=j$, we see again by linearity of the determinant and \eqref{eq:Dhat} that 

\begin{align*}
\oint_\Sigma w^j q_j(w^{-1})f(w)\frac{dw}{2\pi i w}=\frac{\prod_{k=0}^j \frac{\pi \Gamma\left(\frac{\gamma}{2}+k+1\right)}{N^{\frac{\gamma}{2}+k+1}}}{\sqrt{D_{j-1}^{(N)}(F)D_j^{(N)}(F)}}\widehat{D}_j=\frac{D_j^{(N)}(F)}{\sqrt{D_{j-1}^{(N)}(F)D_j^{(N)}(F)}}=\frac{1}{\chi_j}.
\end{align*}

Finally for \eqref{eq:chihat}, we note that from \eqref{eq:qdef} and \eqref{eq:Dhat}

\begin{align*}
\widehat{\chi}_j=\frac{\prod_{k=0}^j \frac{\pi \Gamma(\frac{\gamma}{2}+k+1)}{N^{\frac{\gamma}{2}+k+1}}}{\sqrt{D_{j-1}^{(N)}(F)D_j^{(N)}(F)}}\widehat{D}_{j-1}=\frac{\pi \Gamma(\frac{\gamma}{2}+j+1)}{N^{\frac{\gamma}{2}+j+1}}\frac{D_{j-1}^{(N)}(F)}{\sqrt{D_{j-1}^{(N)}(F)D_j^{(N)}(F)}}=\frac{\pi \Gamma(\frac{\gamma}{2}+j+1)}{N^{\frac{\gamma}{2}+j+1}}\chi_j.
\end{align*}

\end{proof}

\section{Proof of the differential identity}\label{app:RHP}

In this appendix we prove our differential identity -- Lemma \ref{le:DI}. To prove it, we need to recall suitable recursion relations for the polynomials as well as the Christoffel-Darboux identity for the polynomials $p$ and $q$. While these are standard results and the proofs we present below are trivial modifications of those in \cite[Section 2]{DIK1}, there are some cosmetic differences due to the fact that $\chi_j\neq \widehat{\chi}_j$, so we choose to present a proof here. We start with some recurrence relations for the polynomials -- this is very similar to \cite[Lemma 2.2]{DIK1}.

\begin{lemma}\label{le:rec}
Fix a positive integer $n$ and assume that $D_j^{(N)}(F)\neq 0$ for all $j\leq n+1$ $($so that $(p_j)_{j=0}^{n+1}$ and $(q_j)_{j=0}^{n+1}$ exist and each form a basis for the space of polynomials of degree at most $n+1)$. Then the  following identities hold:
\begin{align}
\label{eq:rec1}
\widehat{\chi}_n w p_n(w) &= \widehat{\chi}_{n+1} p_{n+1}(w) - p_{n+1}(0) w^{n+1} q_{n+1}(w^{-1}), \\
\label{eq:rec2}
\chi_n w^{-1} q_n(w^{-1}) &= \chi_{n+1} q_{n+1}(w^{-1}) - q_{n+1}(0) w^{-n-1} p_{n+1}(w),\\
\label{eq:rec3}
\widehat{\chi}_{n+1}w^{-1}q_n(w^{-1}) & = \widehat{\chi}_n q_{n+1}(w^{-1}) - q_{n+1}(0) \frac{\widehat{\chi}_n}{\chi_n}w^{-n} p_n(w), \\
\label{eq:rec4}
\chi_n \widehat{\chi}_n& = \chi_{n+1}\widehat{\chi}_{n+1} - p_{n+1}(0)q_{n+1}(0).
\end{align}
\end{lemma}

\begin{proof}
Let 
\begin{align*}
g(w) := p_n(w) - aw^{-1} p_{n+1}(w) - b w^n q_{n+1}(w^{-1}).
\end{align*}

\noindent We want to choose $a$ and $b$ so that $g$ vanishes. We first show that with a good choice of $b$, $g$ is actually a polynomial in $w$ so that we can express it in terms of the polynomials $p_k$ (with $k\leq n$). We then show that by choosing $a$ the correct way, the coefficients of $p_k$ vanish for all $k\leq n$. 

We thus begin by making sure that the term of order $w^{-1}$ vanishes (there are no lower order terms in $g$). For this, we note that the coefficient of $w^{-1}$ in $g(w)$ is $-ap_{n+1}(0)-b\widehat{\chi}_{n+1}$, so we choose $b = -\frac{ap_{n+1}(0)}{\widehat{\chi}_{n+1}}$. Thus $g$ is a polynomial in $w$ and its degree is at most $n$. To expand it in the basis $(p_k)$,  we know from \eqref{eq:qporto} that is enough to evaluate $\oint_{\Sigma} g(w) q_l(w^{-1})f(w)\frac{dw}{2\pi i w}$ for $l\leq n$. We have from \eqref{eq:qporto}
\begin{itemize}
\item $\displaystyle \oint_\Sigma p_n(w) q_l(w^{-1})f(w)\frac{dw}{2\pi i w} = \delta_{l, n}$.
\item $\displaystyle \oint_\Sigma w^{-1}p_{n+1}(w) q_l(w^{-1})f(w)\frac{dw}{2\pi i w} = \delta_{l, n} \frac{\widehat{\chi}_{n}}{\widehat{\chi}_{n+1}}$.
\item $\displaystyle \oint_\Sigma w^n q_{n+1}(w^{-1}) q_l(w^{-1})f(w)\frac{dw}{2\pi i w} = 0$.
\end{itemize}

\noindent Therefore if we choose $a = \frac{\widehat{\chi}_{n+1}}{\widehat{\chi}_{n}}$, we see that for all $l\leq n$, $\oint_{\Sigma} g(w) q_l(w^{-1})f(w)\frac{dw}{2\pi i w}=0$ implying that $g(w)=0$ for all $w$. This gives \eqref{eq:rec1}. The proof of \eqref{eq:rec2} is similar, and one can obtain \eqref{eq:rec3} by combining the first two recurrence relations. To obtain \eqref{eq:rec4} one inspects the coefficient of $w^{n+1}$ in \eqref{eq:rec1}.
\end{proof}

This lets us prove the Christoffel-Darboux identity.

\begin{lemma}[Christoffel-Darboux]\label{le:cd}
Let $n$ be a positive integer and assume that $D_j^{(N)}(F)\neq 0$ for all $j\leq n$. For any $w, u \ne 0$, we have
\begin{align}\label{eq:cd1}
(1-u^{-1}w) \sum_{k=0}^{n-1} p_k(w) q_k(u^{-1})
= u^{-n}p_n(u) w^n q_n(w^{-1}) - p_n(w)q_n(u^{-1}).
\end{align}

\noindent In particular, for any $w \ne 0$ and $n \in \mathbb{N}$, 
\begin{align}\label{eq:cd2}
\sum_{k=0}^{n-1} p_k(w) q_k(w^{-1}) = -np_n(w) q_n(w^{-1}) + w\left(q_{n}(w^{-1}) \partial_w p_n(w) - p_n(w) \partial_w q_n(w^{-1})\right).
\end{align}
\end{lemma}

\begin{proof}
Using \eqref{eq:rec1} and \eqref{eq:rec3}, we have
\begin{align*}
& u^{-1}w p_k(w) q_k(u^{-1}) 
= \left(w p_k(w)\right)\left(u^{-1} q_k(u^{-1})\right)\\
& = \left[\frac{\widehat{\chi}_{k+1}}{\widehat{\chi}_{k}}p_{k+1}(w) - \frac{p_{k+1}(0)}{\widehat{\chi}_{k}}w^{k+1}q_{k+1}(w^{-1})\right] \left[ \frac{\widehat{\chi}_{k}}{\widehat{\chi}_{k+1}} q_{k+1}(u^{-1}) - \frac{q_{k+1}(0)}{\widehat{\chi}_{k+1}}\frac{\widehat{\chi}_k}{\chi_k} u^{-k} p_k(u)\right],
\end{align*}

\noindent and hence
\begin{align*}
& (1-u^{-1}w) p_k(w) q_k(u^{-1})
= p_k(w) q_k(u^{-1}) - p_{k+1}(w) q_{k+1}(u^{-1})\\
& \qquad \qquad + \left(\frac{w}{u}\right)^{k+1} \left[\frac{q_{k+1}(w^{-1})}{\widehat{\chi}_{k+1}} p_{k+1}(0) u^{k+1} q_{k+1}(u^{-1}) + \frac{u p_k(u)}{\chi_k}q_{k+1}(0) w^{-k-1}p_k(w)\right. \\
& \qquad \qquad \qquad - \left.\frac{p_{k+1}(0) q_{k+1}(0)}{\chi_k \widehat{\chi}_{k+1}} u p_k(u) q_{k+1}(w^{-1})\right].
\end{align*}

\noindent But from \eqref{eq:rec1}, \eqref{eq:rec2} and \eqref{eq:rec4}, we see that
\begin{align*}
\frac{q_{k+1}(w^{-1})}{\widehat{\chi}_{k+1}} p_{k+1}(0){u}^{k+1} q_{k+1}(u^{-1})
& = q_{k+1}(w^{-1})\left[p_{k+1}(u) - \frac{\widehat{\chi}_{k}}{\widehat{\chi}_{k+1}}u p_k(u)\right]\\
\frac{u p_k(u)}{\chi_k}q_{k+1}(0) w^{-k-1}p_k(w)
& = u p_k(u) \left[\frac{\chi_{k+1}}{\chi_k} q_{k+1}(w^{-1})-w^{-1} q_k(w^{-1})\right]\\
-\frac{p_{k+1}(0) q_{k+1}(0)}{\chi_k \widehat{\chi}_{k+1}} u p_k(u) q_{k+1}(w^{-1})
& = \left(\frac{\widehat{\chi}_k}{\widehat{\chi}_{k+1}} - \frac{\chi_{k+1}}{\chi_k}\right) up_k(u) q_{k+1}(w^{-1})
\end{align*}

\noindent and therefore 
\begin{align*}
(1-u^{-1}w) p_k(w) q_k(u^{-1}) &= 
p_k(w) q_k(u^{-1}) - p_{k+1}(w) q_{k+1}(u^{-1})\\
& \quad + \left(\frac{w}{u}\right)^{k+1} p_{k+1}(u)q_{k+1}(w^{-1}) -\left(\frac{w}{u}\right)^{k}p_k(u)q_{k}(w^{-1}).
\end{align*}

\noindent \eqref{eq:cd1} now follows by taking the sum from $k=0$ to $k=n-1$. \eqref{eq:cd2} follows from dividing \eqref{eq:cd1} by $(1-u^{-1}w)$ and letting $u\to w$.
\end{proof}

We can finally turn to our differential identity. This is very similar to corresponding proofs in \cite{DIK1,DIK2,Krasovsky}.

\begin{proof}[Proof of Lemma \ref{le:DI}]
We begin by noting that from Lemma \ref{le:ortop}

\begin{equation}\label{eq:dip1}
\partial_\gamma \log D_{N-1}(F;\gamma)=-2\sum_{j=0}^{N-1}\frac{\partial_\gamma \chi_j}{\chi_j},
\end{equation}

\noindent where the smoothness of $\chi_j$ and $D_{N-1}$ as functions of $\gamma$ follows e.g. from the determinantal representation \eqref{eq:polydet}. It follows from \eqref{eq:qporto} that 

\begin{align*}
\oint_\Sigma \left[ \partial_{\gamma} p_j(w) \right] q_{j}(w^{-1}) f(w) \frac{dw}{2\pi i w} = \frac{\partial_\gamma \chi_j}{\chi_j} \qquad \mathrm{and} \qquad
\oint_\Sigma p_j(w)\left[ \partial_{\gamma}  q_{j}(w^{-1}) \right] f(w) \frac{dw}{2\pi i w}= \frac{\partial_\gamma \widehat{\chi}_j}{\widehat{\chi}_j}.
\end{align*}

\noindent Moreover, we see from \eqref{eq:chihat} that

\begin{align}\label{eq:dchihat}
\frac{\partial_\gamma \widehat{\chi}_j}{\widehat{\chi}_j} = \frac{\partial_\gamma \chi_j}{\chi_j} + \partial_\gamma \log \Gamma \left(\frac{\gamma}{2} + j + 1 \right) - \frac{1}{2} \log N.
\end{align}

\noindent We can thus rewrite \eqref{eq:dip1} as 

\begin{align}\label{eq:dip2}
\partial_\gamma \log D_{N-1}(F;\gamma)&=-\sum_{j=0}^{N-1}\oint_\Sigma \partial_\gamma\left(p_j(w)q_{j-1}(w^{-1})\right)f(w)\frac{dw}{2\pi i w}\\
\notag & \quad +\partial_\gamma \sum_{j=0}^{N-1}\log \frac{\Gamma\left(\frac{\gamma}{2}+j+1\right)}{N^{\frac{\gamma}{2}}}.
\end{align}

Applying the Christoffel-Darboux identity \eqref{eq:cd2} and the orthogonality relations \eqref{eq:qporto}, we have

\begin{align}
\notag \oint_\Sigma & \left[\partial_\gamma \sum_{n=0}^{N-1} p_n(w) q_n(w^{-1})\right] f(w)\frac{dw}{2\pi i w}\\
\notag &\quad =\oint_\Sigma\partial_\gamma\left[-Np_N(w)q_N(w^{-1})+w(q_N(w^{-1})\partial_w p_N(w)-p_N(w)\partial_w q_N(w^{-1}))\right]f(w)\frac{dw}{2\pi i w}\\
\notag & \quad = -N \left[\frac{\partial_\gamma \chi_N}{\chi_N} + \frac{\partial_\gamma \widehat{\chi}_N}{\widehat{\chi}_N}\right] +\oint_\Sigma \left[\partial_\gamma \left(q_N(w^{-1}) \partial_w p_N(w) - p_N(w) \partial_w q_N(w^{-1})\right) \right] wf(w)\frac{dw}{2\pi i w}\\
& \quad = \underbrace{\oint_\Sigma \frac{\partial q_N(w^{-1})}{\partial \gamma} \frac{\partial p_N(w)}{\partial w} f(w)\frac{dw}{2\pi i }}_{=:I_1}
-\underbrace{\oint_\Sigma \frac{\partial q_N(w^{-1})}{\partial w} \frac{\partial p_N(w)}{\partial \gamma} f(w)\frac{dw}{2\pi i }}_{=:I_2}.\label{eq:I1I2}
\end{align}

Let us first evaluate $I_1$. For $\epsilon > 0$, let us write $\Sigma_\epsilon$ for a circular, radius $\epsilon$, indentation of $\Sigma$ at $x$ such that $\Sigma_\epsilon$ does not pass through any point of the interval $[0,x]$. We write $A_\epsilon =  \Sigma_\epsilon \setminus \Sigma$ and $x_\epsilon^{\pm} = x + \epsilon e^{\pm i\theta_\epsilon}$ to be the end points of $A_\epsilon$ in the upper and lower half plane respectively. Since
\begin{align*}
\lim_{\epsilon \to 0} \int_{A_\epsilon} \frac{\partial q_N(w^{-1})}{\partial \gamma} \frac{\partial p_N(w)}{\partial w} f(w)\frac{dw}{2\pi i} = 0
\end{align*}

\noindent (as the singularity is integrable) and
\begin{align*}
\partial_w f(w) = f(w) \left[\frac{\gamma}{2}\frac{1}{w-x} -\frac{\gamma}{2}\frac{1}{w} - Nx \right]=-Nx f(w)+\frac{\gamma}{2}\frac{xf(w)}{w-x}\frac{1}{w},
\end{align*}

\noindent integration by parts gives
\begin{align}\label{eq:I1}
\notag  I_1
& = \lim_{\epsilon \to 0} \int_{\Sigma_\epsilon \setminus A_\epsilon} \frac{\partial q_N(w^{-1})}{\partial \gamma} \frac{\partial p_N(w)}{\partial w} f(w)\frac{dw}{2\pi i}\\
\notag & = Nx\oint_{\Sigma} wp_N(w) \frac{\partial q_N(w^{-1})}{\partial \gamma}f(w) \frac{dw}{2\pi i w}
-\oint_\Sigma p_N(w) w\partial_w \left[ \frac{\partial q_N(w^{-1})}{\partial \gamma}\right]  f(w) \frac{dw}{2\pi i w}\\
& \qquad  - \lim_{\epsilon \to 0} \left[\frac{\gamma}{2} \int_{\Sigma_\epsilon \setminus A_\epsilon} p_N(w) \frac{\partial q_N(w^{-1})}{\partial \gamma}\frac{xf(w)}{w-x} \frac{dw}{2\pi i w} 
-\left.p_N(w) \frac{\partial q_N(w^{-1})}{\partial \gamma}\frac{f(w)}{2\pi i}\right|_{x_{\epsilon}^+}^{x_\epsilon^-}
\right].
\end{align}

\noindent Given that  $wp_N(w)=\frac{\widehat{\chi}_{N+1}}{\widehat{\chi}_N}p_{N+1}(w)-\frac{p_{N+1}(0)}{\widehat{\chi}_N}w^{N+1}q_{N+1}(w^{-1})$ from \eqref{eq:rec1}, we obtain by orthogonality, namely \eqref{eq:ortoc} and \eqref{eq:qporto}, that
\begin{align}\label{eq:I1c}
\oint_{\Sigma}wp_N(w)\frac{\partial q_N(w^{-1})}{\partial\gamma}f(w)\frac{dw}{2\pi i w}&=-\frac{p_{N+1}(0)}{\widehat{\chi}_N}\frac{\partial_\gamma q_N(0)}{\chi_{N+1}}.
\end{align}

\noindent It is also not difficult to verify (from \eqref{eq:qporto}) that
\begin{equation}\label{eq:I1a}
-\oint_\Sigma p_N(w) w\partial_w \left[ \frac{\partial q_N(w^{-1})}{\partial \gamma}\right]  f(w) \frac{dw}{2\pi i w}=N\frac{\partial_\gamma \widehat{\chi}_N}{\widehat{\chi}_N}.
\end{equation}

Next we study
\begin{align*}
 - \frac{\gamma}{2}\int_{\Sigma_\epsilon \setminus A_\epsilon}& p_N(w) \frac{\partial q_N(w^{-1})}{\partial \gamma}\frac{xf(w)}{w-x} \frac{dw}{2\pi i w}\\
&=-\frac{\gamma}{2}x^N\frac{\partial q_N(x^{-1})}{\partial\gamma}\int_{\Sigma_\epsilon \setminus A_\epsilon} p_N(w)w^{-N+1}\frac{f(w)}{w-x}\frac{dw}{2\pi i w} \\
&\quad -\frac{\gamma}{2}\int_{\Sigma_\epsilon \setminus A_\epsilon} p_N(w)w^{-N+1}\frac{\partial_\gamma(w^{N-1}q_N(w^{-1})-x^{N-1}q_N(x^{-1}))}{w-x}xf(w)\frac{dw}{2\pi i w}.
\end{align*}

\noindent As
\begin{align*}
w^{-N+1}\frac{\partial_\gamma(w^{N-1}q_N(w^{-1})-x^{N-1}q_N(x^{-1}))}{w-x}x=-\partial_\gamma \widehat{\chi}_N w^{-N}+\mathcal{P}_{N-1}(w^{-1}),
\end{align*}

\noindent for some polynomial $\mathcal{P}_{N-1}$ of degree at most $N-1$, \eqref{eq:ortoc} implies that

\begin{align}\label{eq:I1b}
-\frac{\gamma}{2}\lim_{\epsilon \to 0} \int_{\Sigma_\epsilon \setminus A_\epsilon} p_N(w)w^{-N+1}\frac{\partial_\gamma(w^{N-1}q_N(w^{-1})-x^{N-1}q_N(x^{-1}))}{w-x}xf(w)\frac{dw}{2\pi i w}
& = \frac{\gamma}{2}\frac{\partial_\gamma \widehat{\chi}_N}{\widehat{\chi}_N}.
\end{align}

Our next goal is to understand the asymptotics of 

\begin{align*}
&  \int_{\Sigma_\epsilon \setminus A_\epsilon} p_N(w)w^{-N+1}f(w) \left[ \frac{1}{w-z} - \frac{1}{w-x}\right]\frac{dw}{2\pi i w}\\
& \qquad = (z-x)  \int_{\Sigma_\epsilon \setminus A_\epsilon} \underbrace{\frac{p_N(w)w^{-N+1} w^{-\frac{\gamma}{2}} e^{-Nxw}}{2\pi i w} }_{=: \tilde{f}(w)}\frac{(w-x)^{\frac{\gamma}{2}}}{(w-z)(w-x)}dw\\
\end{align*}

\noindent in the limit where we first let $\epsilon\to 0$ and then $z \to x$, in $\mathrm{Int}(\Sigma) \cap \{w \in \mathbb{C}: \Im w > 0\}$. Let us write $\Sigma_\epsilon^\pm$ to be the part of $\Sigma_\epsilon \setminus A_\epsilon$ in the upper and lower half planes respectively, and deform $\Sigma_\epsilon^\pm$ into two parts $l_\epsilon^\pm$ and $L_\epsilon^\pm$, where $l_\epsilon^\pm \subset \{x + \kappa e^{\pm i\theta_\epsilon}: \kappa > 0\}$ (see the left diagram in Figure \ref{fig:I1int}).

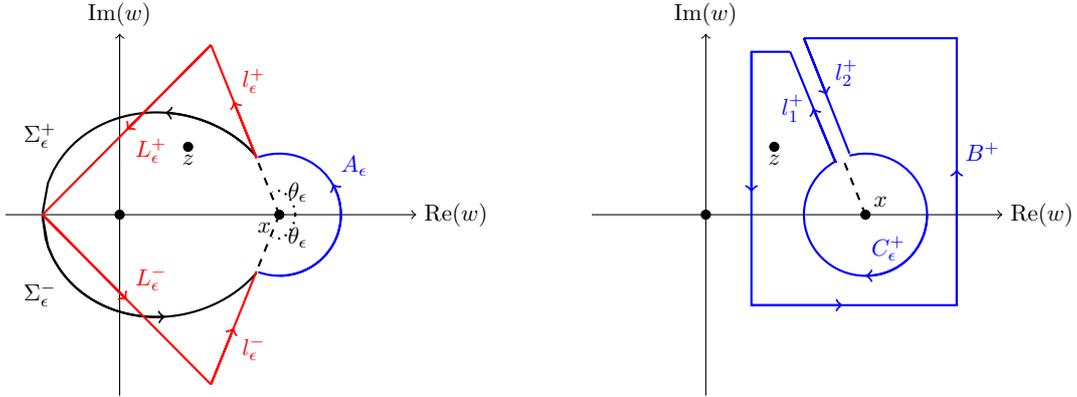
\begin{figure}[h!]
\begin{center}
\begin{tikzpicture}[scale = 3, every node/.style={scale=0.8}]
\draw[->] (-0.5,0) -- (1.3, 0) node[right] {$\mathrm{Re}(w)$};
\draw[->] (0, -0.8) -- (0, 0.8) node[above] {$\mathrm{Im}(w)$};

\draw[thick, <-, domain= 0.2:0.6, samples=40] plot ({\x}, {sqrt(0.49*exp(1.4*(\x-0.7))-\x*\x)});
\draw[thick, domain=-0.338392259:0.6, samples=40] plot ({\x}, {sqrt(0.49*exp(1.4*(\x-0.7))-\x*\x)});
\draw[thick, ->, domain= -0.338392259:0.2, samples=40] plot ({\x}, {-sqrt(0.49*exp(1.4*(\x-0.7))-\x*\x)});
\draw[thick, domain=-0.338392259:0.6, samples=40] plot ({\x}, {-sqrt(0.49*exp(1.4*(\x-0.7))-\x*\x)});

\node at (-0.35, 0.35) {$\Sigma_\epsilon^+$};
\node at (-0.35, -0.35) {$\Sigma_\epsilon^-$};

\draw[blue, thick] (0.97, 0) arc [radius=0.27, start angle=0, end angle= -110] (0.97, 0) arc [radius=0.27, start angle=0, end angle= 110];
\draw[blue, thick, ->] (0.7, -0.27) arc [radius=0.27, start angle=-90, end angle= 30] node[above right] {$A_\epsilon$};
\draw[dashed, thick] (0.7, 0) -- (0.6, 0.25);
\draw[dashed, thick] (0.7, 0) -- (0.6, -0.25);

\draw[dotted, thick] (0.77, 0) to [out = 90, in = 30] (0.67,0.1);
\node at (0.78, 0.1) {$\theta_\epsilon$};
\draw[dotted, thick] (0.77, 0) to [out = -90, in = -30] (0.67, -0.1);
\node at (0.78, -0.1) {$\theta_\epsilon$};

\draw[fill = black] (0.3, 0.3) circle [radius = 0.02] node[below]{$z$};

\draw[fill = black] (0.7, 0) circle [radius = 0.02] node[below left]{$x$};
\draw[fill = black] (0, 0) circle [radius = 0.02];

\draw[red, thick, ->] (-0.3384, 0) -- (0.4, 0.75) (0.4, 0.75) -- (0.0308, 0.375) node[below right] {$L_\epsilon^+$};
\draw[red, thick, ->] (0.6, 0.25) --  (0.4, 0.75) (0.6, 0.25) -- (0.5, 0.5) node [above right] {$l_\epsilon^+$} ;

\draw[red, thick, ->] (-0.3384, 0) -- (0.4, -0.75) (-0.3384, 0) -- (0.0308, -0.375) node[above right] {$L_\epsilon^-$};
\draw[red, thick, ->] (0.6, -0.25) --  (0.4, -0.75) (0.4, -0.75) -- (0.5, -0.5) node [below right] {$l_\epsilon^-$} ;

\end{tikzpicture}
\hspace{1cm}
\begin{tikzpicture}[scale = 3, every node/.style={scale=0.8}]
\draw[->] (-0.5,0) -- (1.3, 0) node[right] {$\mathrm{Re}(w)$};
\draw[->] (0, -0.8) -- (0, 0.8) node[above] {$\mathrm{Im}(w)$};

\draw[blue, thick] (0.97, 0) arc [radius = 0.27, start angle=0, end angle=105];
\draw[blue, thick] (0.97, 0) arc [radius=0.27, start angle=0, end angle= -242];
\draw[blue, thick,->] (0.97, 0) arc [radius=0.27, start angle=0, end angle= -90];
\node at (0.8, -0.15) {\color{blue} $C^+_\epsilon$};

\draw[dashed, thick] (0.7, 0) -- (0.6, 0.25);

\draw[blue, thick, ->] (0.57, 0.23) --  (0.37, 0.72)(0.57, 0.23) -- (0.47, 0.47) node[ left] {$l_1^+$};
\draw[blue, thick, ->] (0.63, 0.27) --  (0.43, 0.78)  (0.43, 0.78) --(0.53, 0.53) node[above right] {$l_2^+$};

\draw[blue, thick] (0.37, 0.72) -- (0.2, 0.72) -- (0.2, -0.4) -- (1.1, -0.4) -- (1.1, 0.78) -- (0.43, 0.78);
\draw[blue, thick, ->] (0.2, 0.72) -- (0.2, 0.1);
\draw[blue, thick, ->] (0.2, -0.4) -- (0.6, -0.4);
\draw[blue, thick, ->] (1.1, -0.4) -- (1.1, 0.2) node [above right] {$B^+$};

\draw[fill = black] (0.3, 0.3) circle [radius = 0.02] node[below]{$z$};

\draw[fill = black] (0.7, 0) circle [radius = 0.02] node[above right]{$x$};
\draw[fill = black] (0, 0) circle [radius = 0.02];
\end{tikzpicture}
\end{center}
\caption{\label{fig:I1int} Integration contour.}
\end{figure}

In order to evaluate
\begin{align*}
\int_{l_\epsilon^+} \frac{\tilde{f}(w)(w-x)^{\frac{\gamma}{2}}}{(w-z)(w-x)}dw,
\end{align*}

\noindent we shall consider the contour $\mathcal{C}^+ := l_1^+ \cup l_2^+ \cup B^+ \cup C_\epsilon^+$ (see the right diagram in Figure \ref{fig:I1int}). Let $(w-x)_1^{\frac{\gamma}{2}}$ be such that the branch cut is given by $\{x + \kappa e^{i \theta_\epsilon}: \kappa > 0\}$. Then by the residue theorem we have
\begin{align} \label{eq:I1C1}
\int_{\mathcal{C}^+} \frac{\tilde{f}(w)(w-x)_1^{\frac{\gamma}{2}}}{(w-z)(w-x)}dw 
= 2 \pi i \tilde{f}(z)(z-x)_1^{\frac{\gamma}{2}-1}
= 2 \pi i \tilde{f}(z)e^{-i \pi \gamma}(z-x)^{\frac{\gamma}{2}-1}
\end{align}

\noindent since $(z-x)_1^{\frac{\gamma}{2}}=e^{-i \pi \gamma}(z-x)^{\frac{\gamma}{2}}$ for our choice of $z$. On the other hand
\begin{align}\label{eq:I1C2}
\notag & \int_{\mathcal{C}^+} \frac{\tilde{f}(w)(w-x)_1^{\frac{\gamma}{2}}}{(w-z)(w-x)}dw\\
\notag & = \int_{l_1^+ \cup l_2^+} \frac{\tilde{f}(w)(w-x)_1^{\frac{\gamma}{2}}}{(w-z)(w-x)}dw
+ \underbrace{\int_{B^+} \frac{\tilde{f}(w)(w-x)_1^{\frac{\gamma}{2}}}{(w-z)(w-x)}dw}_{=: F_1(z)}
+ \int_{C_\epsilon^+} \frac{\tilde{f}(w)(w-x)_1^{\frac{\gamma}{2}}}{(w-z)(w-x)}dw\\
& = (e^{-i \pi \gamma} - 1) \int_{l_\epsilon^+} \frac{\tilde{f}(w)(w-x)^{\frac{\gamma}{2}}}{(w-z)(w-x)}dw  + F_1(z)
+ \frac{\tilde{f}(x)}{x-z} \frac{\epsilon^{\frac{\gamma}{2}}}{\gamma / 2} e^{i \theta_\epsilon \gamma / 2}(e^{-i \pi \gamma} - 1) 
+ \frac{\mathcal{O}(\epsilon^{\frac{\gamma}{2}+1})}{z-x}.
\end{align}

\noindent Comparing \eqref{eq:I1C1} and \eqref{eq:I1C2}, we obtain
\begin{align}
\notag & \int_{l_\epsilon^+} \frac{\tilde{f}(w)(w-x)^{\frac{\gamma}{2}}}{(w-z)(w-x)}dw\\
& \qquad = \frac{2 \pi i e^{-i \pi \gamma}}{(e^{-i \pi \gamma} - 1)} \tilde{f}(z)(z-x)^{\frac{\gamma}{2}-1} 
- \frac{\tilde{f}(x)}{x-z} \frac{\epsilon^{\frac{\gamma}{2}}}{\gamma / 2} e^{i \theta_\epsilon \gamma / 2}
- \frac{F_1(z)+\mathcal{O}(\epsilon^{\frac{\gamma}{2}+1})/ (z-x)}{e^{-i \pi \gamma }-1}
\end{align}

\noindent By choosing a different branch cut and a similar contour integral, we can deduce that
\begin{align}
\notag & \int_{l_\epsilon^-} \frac{\tilde{f}(w)(w-x)^{\frac{\gamma}{2}}}{(w-z)(w-x)}dw\\
& \qquad = \frac{2 \pi i}{(e^{i \pi \gamma} - 1)} \tilde{f}(z)(z-x)^{\frac{\gamma}{2}-1} 
+ \frac{\tilde{f}(x)}{x-z} \frac{\epsilon^{\frac{\gamma}{2}}}{\gamma / 2} e^{-i \theta_\epsilon \gamma / 2}
- \frac{F_2(z)+\mathcal{O}(\epsilon^{\frac{\gamma}{2}+1})/ (z-x)}{e^{i \pi \gamma }-1}
\end{align}

\noindent for some $F_2(z)$. Note that as $z \to x$, both $F_1$ and $F_2$ are bounded and $\tilde{f}(z) (z-x)^{\frac{\gamma}{2}-1} = \tilde{f}(x)(z-x)^{\frac{\gamma}{2}-1} + \mathcal{O}((z-x)^{\frac{\Re \gamma}{2}})$.  Collecting everything, we find

\begin{align*}
\int_{\Sigma_\epsilon \setminus A_\epsilon} p_N(w)&w^{-N+1}f(w) \left[ \frac{1}{w-z} - \frac{1}{w-x}\right]\frac{dw}{2\pi i w}\\
&=\frac{2}{\gamma}\widetilde{f}(x)\epsilon^{\gamma/2}\left(e^{i\theta_\epsilon\gamma/2}-e^{-i\theta_\epsilon\gamma/2}\right)+\mathcal{O}(z-x)+\mathcal{O}(\epsilon^{\frac{\gamma}{2}+1}),
\end{align*}

\noindent for $z\in \mathrm{Int}(\Sigma)$ and $\mathrm{Im}(z)>0$. Moreover, $\mathcal{O}(z-x)$ is uniform in $\epsilon>0$ (recall we take $\epsilon\to 0$ first). Letting $\epsilon\to 0$ and then $z\to x$, we conclude that

\begin{align}
\notag &  \lim_{z \to x} \int_{\Sigma} p_N(w) w^{-N+1} \frac{f(w)}{w-z}\frac{dw}{2\pi i w}\\
\label{eq:regint}
& \qquad = \lim_{\epsilon \to 0} \left[\int_{\Sigma_\epsilon \setminus A_\epsilon} p_N(w) w^{-N+1} \frac{f(w)}{w-x}\frac{dw}{2\pi i w}
+ \tilde{f}(x) \frac{\epsilon^{\frac{\gamma}{2}}}{\gamma/2} 2i \sin \frac{\theta_\epsilon \gamma}{2}\right].
\end{align}

\noindent The same contours may be used to study the case where $\Im(z) < 0$. One can carry out the same procedure, though with minor differences e.g. in the analogue of \eqref{eq:I1C1}, since   $(z-x)_1^{\frac{\gamma}{2}-1} = (z-x)^{\frac{\gamma}{2}-1}$ when $\Im(z) < 0$. Nevertheless, one still ends up with \eqref{eq:regint} when the limit of the Cauchy transform is taken along any sequence $z \in \mathrm{Int}(\Sigma) \setminus [0,x]$ with $\Im(z) < 0$. Therefore \eqref{eq:regint} remains valid whenever the limit $z \to x$ is taken inside the set $\mathrm{Int}(\Sigma) \setminus [0,x]$.

Finally, using the fact that $p_N(w), q_N(w^{-1})$ and $w^{-\frac{\gamma}{2}}e^{-Nxw}$ are analytic near $w = x$,
\begin{align} \label{eq:I1IBPend}
\left.p_N(w) \frac{\partial q_N(w^{-1})}{\partial \gamma}\frac{f(w)}{2\pi i}\right|_{x_{\epsilon}^+}^{x_\epsilon^-}
\notag & = p_N(x) \frac{\partial q_N(x^{-1})}{\partial \gamma}x^{-\frac{\gamma}{2}}e^{-Nx^2} \epsilon^{\frac{\gamma}{2}}(e^{-i \frac{\theta_\epsilon \gamma}{2}} - e^{i \frac{\theta_\epsilon \gamma}{2}}) +\mathcal{O}(\epsilon^{\frac{\Re \gamma}{2} + 1})\\
& = - \frac{\gamma}{2} x^N \frac{\partial q_N(x^{-1})}{\partial \gamma}\tilde{f}(x) \frac{\epsilon^{\frac{\gamma}{2}}}{\gamma/2} 2i \sin \frac{\theta_\epsilon \gamma}{2} + o(1)
\end{align}

\noindent and therefore
\begin{align} \label{eq:I1final}
I_1
\notag & = \left( N + \frac{\gamma}{2}\right)\frac{\partial_\gamma \widehat{\chi}_N}{\widehat{\chi}_N} -Nx\frac{p_{N+1}(0)}{\widehat{\chi}_N}\frac{\partial_\gamma q_N(0)}{\chi_{N+1}}\\
& \qquad - \frac{\gamma}{2} x^N \frac{\partial q_N(x^{-1})}{\partial \gamma} \lim_{z \to x} \int_{\Sigma} p_N(w) w^{-N+1} \frac{f(w)}{w-z}\frac{dw}{2\pi i w}.
\end{align}

As for $I_2$, the identity 
\begin{equation*}
w\partial_\gamma p_N(w)=\frac{\partial_\gamma\chi_N}{\chi_{N+1}}p_{N+1}(w)+\left(\frac{\partial_\gamma \kappa_N}{\chi_N}-\frac{\partial_\gamma\chi_N}{\chi_{N+1}}\frac{\kappa_{N+1}}{\chi_N}\right)p_N(w)+\mathcal{O}(w^{N-1}).
\end{equation*}

\noindent combined with the arguments above leads to

\begin{align}\label{eq:I2final}
I_2
&=-N\frac{\partial_\gamma \chi_N}{\chi_N}+Nx\left(\frac{\partial_\gamma \kappa_N}{\chi_N}-\frac{\partial_\gamma \chi_N}{\chi_N}\frac{\kappa_{N+1}}{\chi_{N+1}}\right)  -\frac{\gamma}{2}\frac{\partial p_N(x)}{\partial\gamma}x \lim_{z \to x} \oint_\Sigma \frac{q_N(w^{-1})f(w)}{w-z}\frac{dw}{2\pi i w}
\end{align}

\noindent where again the limit is taken along any sequence $z \in \mathrm{Int}(\Sigma) \setminus [0,x]$. Gathering all terms, we conclude that
\begin{align*}
& \partial_\gamma\log D_{N-1}(F;\gamma)\\
&=-\left(N+\frac{\gamma}{2}\right)\frac{\partial_\gamma\widehat{\chi}_N}{\widehat{\chi}_N}+\frac{\gamma}{2}x^N\frac{\partial q_N(x^{-1})}{\partial\gamma}\lim_{z \to x}\oint_\Sigma p_N(w)w^{-N+1}\frac{f(w)}{w-z}\frac{dw}{2\pi i w} -Nx\frac{p_{N+1}(0)}{\chi_{N+1}}\frac{\partial_\gamma q_N(0)}{\widehat{\chi}_N}\\
&\quad -N\frac{\partial_\gamma \chi_N}{\chi_N}-\frac{\gamma}{2}\frac{\partial p_N(x)}{\partial\gamma}x\lim_{z \to x}\oint_\Sigma \frac{q_N(w^{-1})f(w)}{w-z}\frac{dw}{2\pi i w}+Nx\left(\frac{\partial_\gamma \kappa_N}{\chi_N}-\frac{\partial_\gamma \chi_N}{\chi_N}\frac{\kappa_{N+1}}{\chi_{N+1}}\right)\\
&\quad +\partial_\gamma \sum_{j=0}^{N-1}\log \frac{\Gamma\left(\frac{\gamma}{2}+j+1\right)}{N^{\frac{\gamma}{2}}},
\end{align*}

\noindent which is the claim.

\end{proof}

\section{Asymptotic analysis of the Riemann-Hilbert problem -- Proofs for Section \ref{sec:RHPNasy}}\label{app:RHPasy}

In this appendix, we give proofs related to the asymptotic analysis of our Riemann-Hilbert problem. We begin with Lemma \ref{le:Sigmasign}.

\begin{proof}[Proof of Lemma \ref{le:Sigmasign}]
The fact that $\Sigma$ is a smooth, simple closed loop, encircling $[0,x]$ and passing only through $x$ follows e.g. from writing 

\begin{equation*}
\Sigma=\left\lbrace\left(u,\sqrt{x^2 e^{2x(u-x)}-u^2}\right): u_0\leq u\leq x\right\rbrace \cup \left\lbrace\left(u,-\sqrt{x^2 e^{2x(u-x)}-u^2}\right): u_0\leq u\leq x\right\rbrace,
\end{equation*}

\noindent where $u_0$ is the unique negative solution to the equation $x(u-x)+\log x-\log |u|$ (one can easily check that this equation has only one negative solution and for $u\in(0,x]$, the only solution is $u=x$). The fact that $\Sigma$ is inside the unit circle is obvious from \eqref{eq:Sigmadef} -- the definition of $\Sigma$.

The fact that $\mathrm{Re}(xw+\ell-\log w)$ is positive in $\mathrm{Int}(\Sigma)$ follows from the definition of $\Sigma$ and evaluating $\mathrm{Re}(xw+\ell-\log w)$ at $w=0$ (recall that $\Sigma$ encircles $[0,x]$). To see that $\mathrm{Re}(xw+\ell-\log w)$ is negative in $\mathrm{Ext}(\Sigma)\cap \lbrace |w|\leq 1\rbrace$, note first that on the unit circle, $\mathrm{Re}(xw+\ell-\log w)=x\mathrm{Re}(w)+\ell\leq x+\log x-x^2<0$ for $x<1$ (this also proves the claim of the uniform negative bound on the unit circle). Then we note that as $\mathrm{Re}(xw+\ell-\log w)$ is zero on $\Sigma$ and its only critical point is $w=1/x>1$, there can't be any points in $\mathrm{Ext}(\Sigma)\cap \lbrace |w|\leq 1\rbrace$ where it's positive (one of them would have to be a critical point).
\end{proof}

Let us then move on to the Riemann-Hilbert problem that $S$ satisfies.

\begin{proof}[Proof of Lemma \ref{le:SRHP}]
Analyticity and continuity of boundary values is clear from the corresponding properties for $Y$ and the definition of $S$. The jump condition across $\lbrace |w|=1\rbrace$ is also immediate from the definitions. Consider then the jump across $\Sigma$. From the definition of $S$ and $T$, we have for $w\in \Sigma\setminus \lbrace x\rbrace$

\begin{align*}
S_+(w)&=S_-(w)\begin{pmatrix}
1 & 0\\
-w^{\frac{\gamma}{2}}(w-x)^{-\frac{\gamma}{2}}e^{-kxw}e^{(N+k)(xw+\ell-\log w)} & 1
\end{pmatrix}\\
&\quad \times  \begin{pmatrix}
e^{-(N+k)(xw+\ell-\log w)} & (w-x)^{\frac{\gamma}{2}} w^{-\frac{\gamma}{2}}e^{kxw} \\
0 & e^{(N+k)(xw+\ell-\log w)}
\end{pmatrix}\\
&\quad \times \begin{pmatrix}
1 & 0\\
-w^{\frac{\gamma}{2}}(w-x)^{-\frac{\gamma}{2}}e^{-kxw}e^{-(N+k)(xw+\ell-\log w)} & 1
\end{pmatrix}\\
&=S_-(w)\begin{pmatrix}
0 & (w-x)^{\frac{\gamma}{2}} w^{-\frac{\gamma}{2}}e^{kxw}\\
-(w-x)^{-\frac{\gamma}{2}}w^{\frac{\gamma}{2}}e^{-kxw} & 0
\end{pmatrix}.
\end{align*}

For the jump across $(0,x)$, note that the only term contributing to the branch cut is $(w-x)^{-\frac{\gamma}{2}}$. The claimed jump is easily obtained by looking at the jump of this function.

For the behavior near zero, we note that (as $N+k+\frac{\gamma}{2}\geq 1+\frac{\gamma}{2}>0$)

$$
w^{\frac{\gamma}{2}}(w-x)^{-\frac{\gamma}{2}}e^{-kxw}e^{-(N+k)(xw+\ell-\log w)}\to 0
$$

\noindent as $w\to 0$ so since $Y(0)$ exists, one sees that $T(0)$ and thus $S(0)$ exist as well. For the behavior at $x$, note first that as $w\to x$, $T(w)$ has the same asymptotic behavior as $Y$: namely \eqref{eq:Yatxasy1}. Thus as $w\to x$ $($off of $C)$,

$$
S(w)=\begin{pmatrix}
\mathcal{O}(1) & \mathcal{O}(1)+\mathcal{O}\left(|w-x|^{\mathrm{Re}(\gamma/2)}\right)\\
\mathcal{O}(1) & \mathcal{O}(1)+\mathcal{O}\left(|w-x|^{\mathrm{Re}(\gamma/2)}\right)
\end{pmatrix}\begin{pmatrix}
1 & 0\\
\mathcal{O}\left(|w-x|^{-\frac{\mathrm{Re}(\gamma)}{2}}\right) & 1
\end{pmatrix},
$$

\noindent from which the claim follows.

 The normalization at infinity is a consequence of the corresponding property for $T$.
\end{proof}

Our next task is to prove that $P^{(x,r)}$ satisfies the RHP we claimed.

\begin{proof}[Proof of Lemma \ref{le:localRHP}]
Let us begin by noting that we can write 

\begin{align*}
P^{(x,r)}(w)&=\begin{pmatrix}
1 & w^{\frac{\gamma}{2}}(w-x)^{-\frac{\gamma}{2}}e^{kxw}e^{\zeta(w)}\frac{\Gamma\left(\frac{\gamma}{2},\zeta(w)\right)}{\Gamma(\frac{\gamma}{2})}\\
0 & 1
\end{pmatrix}\\
&\quad \times \begin{pmatrix}
1 & h_r(w;\gamma)-e^{kxw}w^{\gamma/2}(w-x)^{-\gamma/2}\zeta(w)^{\gamma/2}\sum_{j=0}^{r}\frac{1}{\Gamma(\frac{\gamma}{2}-j)}\zeta(w)^{-j-1}\\
0 & 1
\end{pmatrix}P^{(\infty)}(w).
\end{align*}

\noindent By the definition of the incomplete gamma function (see \eqref{eq:incomp}), the first matrix here has a branch cut along $(-\infty,x)$, but no other singularities. By the definition of $h_r$ (Definition \ref{def:h}), the second matrix is analytic in $U$. From this, we conclude that indeed $P^{(x,r)}$ is analytic in $U\setminus (\Sigma\cup[0,x])$. Continuity of the boundary values is immediate from the definitions. For the jump conditions, we note we just argued that on $\Sigma\setminus \lbrace x\rbrace$, the only jump comes from $P^{(\infty)}$, and as mentioned in Section \ref{sec:global}, one can check easily that it satisfies \eqref{eq:Sjump2}.

For the jump across $(0,x)$, we see that the only contribution to the jump comes from the incomplete gamma function term, and we simply need the following calculation (which is easy to check from \eqref{eq:incomp}):

\begin{align*}
&\left(w^{\frac{\gamma}{2}}(w-x)^{-\frac{\gamma}{2}}e^{kxw}\frac{\Gamma(\frac{\gamma}{2},\zeta(w))e^{\zeta(w)}}{\Gamma(\frac{\gamma}{2})}\right)_+-\left(w^{\frac{\gamma}{2}}(w-x)^{-\frac{\gamma}{2}}e^{kxw}\frac{\Gamma(\frac{\gamma}{2},\zeta(w))e^{\zeta(w)}}{\Gamma(\frac{\gamma}{2})}\right)_-\\
&\quad =w^{\frac{\gamma}{2}}e^{\zeta(w)}e^{kxw}\left[(w-x)^{-\frac{\gamma}{2}}_+-(w-x)^{-\frac{\gamma}{2}}_-\right]\\
&\quad =-2i \sin \frac{\pi \gamma}{2} |w|^{\frac{\gamma}{2}}|w-x|^{-\frac{\gamma}{2}}e^{kxw}e^{\zeta(w)}
\end{align*}

\noindent from which we find that for $w\in(0,x)\cap U$

\begin{align*}
\left[P^{(x,r)}_-(w)\right]^{-1}P^{(x,r)}_+(w)&=\begin{pmatrix}
0 & -e^{kxw}\\
e^{-kxw} & 0
\end{pmatrix}\begin{pmatrix}
1 & -2i \sin \frac{\pi \gamma}{2} |w|^{\frac{\gamma}{2}}|w-x|^{-\frac{\gamma}{2}}e^{kxw}e^{\zeta(w)}\\
0 & 1
\end{pmatrix}\\
&\quad \times \begin{pmatrix}
0 & e^{kxw}\\
-e^{-kxw} & 0
\end{pmatrix}\\
&=\begin{pmatrix}
1 & 0\\
2i \sin \frac{\pi \gamma}{2} |w|^{\frac{\gamma}{2}}|w-x|^{-\frac{\gamma}{2}}e^{-kxw}e^{\zeta(w)} & 1
\end{pmatrix},
\end{align*}

\noindent which is \eqref{eq:localjump2}.

Let us then move onto the behavior at $x$. We begin with \eqref{eq:localasy}. Simply using the definition of $S$, $\zeta$, $\widehat{P}^{(\infty,r)}$, and $P^{(x,r)}$, we see that for $w\in \mathrm{Int}(\Sigma)\setminus[0,x]$

\begin{align*}
S(w)P^{(x,r)}(w)^{-1}&=T(w)
\begin{pmatrix}
1 & 0\\
-w^{\frac{\gamma}{2}}(w-x)^{-\frac{\gamma}{2}}e^{-kxw}e^{-(N+k)(xw+\ell-\log w)} & 1
\end{pmatrix}\begin{pmatrix}
0 & -e^{kxw}\\
e^{-kxw} & 0
\end{pmatrix} \\
&\quad \times\begin{pmatrix}
1 & -h_r(w;\gamma)-Q_r(w)\\
0 & 1
\end{pmatrix}\\
&=T(w)\begin{pmatrix}
0 & -e^{kxw}\\
e^{-kxw} & w^{\frac{\gamma}{2}}(w-x)^{-\frac{\gamma}{2}}e^{\zeta(w)}
\end{pmatrix}\begin{pmatrix}
1 & -h_r(w;\gamma)-Q_r(w)\\
0 & 1
\end{pmatrix}\\
&=T(w)\begin{pmatrix}
0 & -e^{kxw}\\
e^{-kxw} & \alpha_1(w)+\alpha_2(w)
\end{pmatrix}
\end{align*}

\noindent where 

$$
\alpha_1(w)=w^{\gamma/2}(w-x)^{-\gamma/2}e^{\zeta(w)}\left[\frac{\Gamma(\frac{\gamma}{2},\zeta(w))}{\Gamma(\frac{\gamma}{2})}-1\right]
$$

\noindent and 

$$
\alpha_2(w)=w^{\gamma/2}(w-x)^{-\gamma/2}\zeta(w)^{\gamma/2}\sum_{j=0}^{r} \frac{1}{\Gamma(\frac{\gamma}{2}-j)}\zeta(w)^{-1-j}-e^{-kxw}h_r(w;\gamma).
$$

\noindent From the definition of $h_r$, $\alpha_2$ is analytic, while from the definition of the incomplete gamma function, we see that 

\begin{align*}
\alpha_1(w)=-w^{\gamma/2}(w-x)^{-\gamma/2}e^{\zeta(w)}\zeta(w)^{\gamma/2}\gamma^*\left(\frac{\gamma}{2},\zeta(w)\right).
\end{align*}

\noindent Again, as $\zeta(w)$ has a simple zero at $x$, this is also analytic, so we have 

$$
S(w)=T(w)\begin{pmatrix}
0 & \mathcal{O}(1)\\
\mathcal{O}(1) & \mathcal{O}(1)
\end{pmatrix}
$$

\noindent from which the claim follows once one notices that the definition of $T$ implies that it has the same asymptotic behavior as $Y$ at $x$. For \eqref{eq:localasyb}, we note that for $w\in\mathrm{Ext}(\Sigma)$, 

\begin{align*}
P^{(x.r)}(w)&=\begin{pmatrix}
1 & Q_r(w)+h_r(w)\\
0 & 1
\end{pmatrix}\begin{pmatrix}
w^{\gamma/2}(w-x)^{-\gamma/2} & 0\\
0 & w^{-\gamma/2}(w-x)^{\gamma/2}
\end{pmatrix}\\
&=\begin{pmatrix}
w^{\gamma/2}(w-x)^{-\gamma/2} & (Q_r(w)+h_r(w)) w^{-\gamma/2}(w-x)^{\gamma/2}\\
0 & w^{-\gamma/2}(w-x)^{\gamma/2}
\end{pmatrix}.
\end{align*}

\noindent Again, 

\begin{align*}
Q_r(w)+h_r(w)&=h_r(w)-e^{kxw}w^{\gamma/2}(w-x)^{-\gamma/2}\zeta(w)^{\gamma/2}\sum_{j=0}^{r}\frac{1}{\Gamma(\frac{\gamma}{2}-j)}\zeta(w)^{-j-1}\\
&\quad +e^{kxw}w^{\gamma/2}(w-x)^{-\gamma/2}e^{\zeta(w)}\frac{\Gamma(\frac{\gamma}{2},\zeta(w))}{\Gamma(\frac{\gamma}{2})},
\end{align*}

\noindent where by the definition of $h_r$ (Definition \ref{def:h}), the first row of this equation is bounded at $x$ and by the definition of the incomplete gamma function -- namely \eqref{eq:incomp}, the second row of this equation is $\mathcal{O}(1)+\mathcal{O}(|w-x|^{-\frac{\mathrm{Re}(\gamma)}{2}})$. Putting everything together, we find (a stronger claim than) \eqref{eq:localasyb}.

\smallskip

Finally, we need to check the matching condition \eqref{eq:localmatch}. By the definition of $P^{(x,r)}$, we find immediately that for any $w\in U\setminus (\Sigma\cup[0,x])$

$$
P^{(x,r)}(w)\widehat{P}^{(\infty,r)}(w)^{-1}=I+\begin{pmatrix}
0 & Q_r(w)\\
0 & 0
\end{pmatrix}.
$$

\noindent Now for $w\in \partial U$, $|\zeta(w)|\asymp N$ uniformly in $w\in \partial U$ ($a\asymp b$ meaning $a=\mathcal{O}(b)$ and $b=\mathcal{O}(a)$) so we need to find the large $|\zeta|$ asymptotics of $Q$. For this, we use the following asymptotic expansion of the incomplete gamma function  (see e.g. \cite[Section 4.2]{olver}, where the proof is for real $\gamma$, but it works with obvious modifications also for complex $\gamma$): for any $p\in \Z_+$,

\begin{align}\label{eq:gammaasy}
\frac{\Gamma(\frac{\gamma}{2},\zeta)}{\Gamma\left(\frac{\gamma}{2}\right)}e^\zeta=\zeta^{\frac{\gamma}{2}-1}\left(\sum_{k=0}^p \frac{1}{\Gamma(\frac{\gamma}{2}-k)}\zeta^{-k}+\mathcal{O}(\zeta^{-p-1})\right),
\end{align}

\noindent where the error is uniform in $\gamma$ in compact subsets of  $\lbrace \gamma\in \C:\mathrm{Re}(\gamma)>-2\rbrace$. This yields immediately that uniformly in $w\in\partial U$ (and uniformly in the relevant $\gamma$ and $x$)

$$
Q_r(w)=\mathcal{O}(|\zeta(w)|^{\frac{\gamma}{2}-r-2})=\mathcal{O}(N^{\frac{\gamma}{2}-r-2}),
$$

\noindent which implies \eqref{eq:localjump1} and concludes the proof.

\end{proof}

We now turn to proving that $R$ is a solution to the RHP we claimed.

\begin{proof}[Proof of Lemma \ref{le:RRHP}]
The proof is largely standard. Uniqueness is the standard argument. We note that by construction, the branch cuts of the parametrices cancel with those of $S$, and the only jumps are across $\partial U$ and the unit circle. For analyticity, one still needs to check that there is no isolated singularity at $x$. Using \eqref{eq:Sasy}, \eqref{eq:localasy}, and \eqref{eq:localasyb}, one sees that any possible singularity of $R$ at $x$ is of bounded degree and can't thus be essential. Note that if there were a pole, then independently from the direction $w$ approaches $x$ from, one would have that for some positive integer $m$, $(w-x)^m R(w)$  would converge to a finite non-zero matrix as $w\to x$. Now if we approach from $\mathrm{Int}(\Sigma)\setminus[0,x]$, then by \eqref{eq:localasy} (and the fact that $\mathrm{Re}(\gamma)>-2$), $(w-x)R(w)\to 0$ so we can't have a pole -- $R$ is analytic in the claimed region.

Continuity of the boundary values and the structure of the jump matrices follow directly from the relevant definitions and from \eqref{eq:Sjump1}. The normalization at infinity also follows from the asymptotic behavior of $S$ and $\widehat{P}^{(\infty,r)}$ at infinity. The estimates for the jump matrices follow from \eqref{eq:localmatch} and Lemma \ref{le:Sigmasign}.

\end{proof}

We conclude with the proof of the asymptotic behavior of $R$.

\begin{proof}[Proof of Lemma \ref{le:smnorm}]
Again, most of the proof is standard and surely obvious for experts, but for the convenience of the reader, we offer a sketch of a proof here. We follow reasoning from \cite{Krasovsky,DIK1,DIK2}. We now recall how one sees that a unique solution exists for this RHP. Again, uniqueness can be proven the standard way. To see existence, we introduce some (standard) notation: for $w\in \C\setminus \Gamma_R$, let

\begin{align*}
C(f) :=  \int_{\Gamma_R} \frac{f(s)}{s-w}\frac{ds}{2\pi i}
\end{align*}

\noindent and let $C_{\Delta_R}(f) = C_-(f \Delta_R)$, where $C_-(f)(w) = \lim_{z \to w} C(f)(z)$ as $z$ approaches $w \in \Gamma_R$ from the $-$ side of $\Gamma_R$. Since $C_-: L^2(\Gamma_R) \to L^2(\Gamma_R)$ is a bounded operator (see \cite[Appendix A]{dkmlvz} and the references therein), our estimate on the jump matrix of $R$, namely $||\Delta_R||_{L^\infty(\Gamma_R)}=\mathcal{O}(N^{\frac{1}{2}\mathrm{Re}(\gamma)-r-2})$ implies that the operator norm of $C_{\Delta_R}$ is  $\mathcal{O}(N^{\frac{\mathrm{Re}(\gamma)}{2}-r-2})$, and therefore for large enough $N$ and choosing $r$ suitably ("large enough" and $r$ depending only on the compact set $K$ that  $\gamma$ is in and the compact subset of $(0,1)$ that $x$ is in), $I-C_{\Delta_R}$ is invertible. Arguing as in \cite[the proof of Theorem 7.8]{dkmlvz} (though in a slightly inverted order since we don't know the existence of a solution) one can check that 

\begin{align}\label{eq:Rsol}
R = I + C[\Delta_R + (I - C_{\Delta_R})^{-1} (C_{\Delta_R}(I))\Delta_R]
\end{align}

\noindent is a solution to the problem. Moreover, one can check that this implies that $R$ can also be represented in terms of its boundary values:

\begin{align}\label{eq:Rsol2}
R(w)=I+(C_{\Delta_R}R_-)(w)=I+\int_{\Gamma_R}\frac{R_-(s)\Delta_R(s)}{s-w}\frac{ds}{2\pi i}.
\end{align}

\noindent To get a hold of the asymptotic behavior of $R$, we note one consequence of the definition \eqref{eq:Rsol} is that $R_--I=(1-C_{\Delta_R})^{-1}C_{\Delta_R}(I)$. Since the norm of $C_{\Delta_R}$ is of order $N^{\frac{1}{2}\mathrm{Re}(\gamma)-r-2}$, we see from this that

\begin{align}\label{eq:Rbound}
||R_--I||_{L^2(\Gamma_R)}\leq ||(I-C_{\Delta_R})^{-1}||_{L^2(\Gamma_R)\to L^2(\Gamma_R)}||C_{\Delta_R}(I)||_{L^2(\Gamma_R)}=\mathcal{O}\left(N^{\frac{1}{2}\mathrm{Re}(\gamma)-r-2}\right).
\end{align}

\noindent Let us now fix $\delta>0$ and let $w$ be at distance at least $\delta$ from $\Gamma_R$. Then applying \eqref{eq:Rbound} to \eqref{eq:Rsol2} and using Cauchy-Schwarz, we see that 

\begin{align*}
|R(w)-I|&\leq |(C_{\Delta_R}I)(w)|+|(C_{\Delta_R}[R_--I])(w)|\\
&=\mathcal{O}\left(||{\Delta_R}||_{L^\infty(\Gamma_R)}\right)+\mathcal{O}\left(||R_--I||_{L^2(\Gamma_R)}||\Delta_R||_{L^2(\Gamma_R)}\right)\\
&=\mathcal{O}\left(N^{\frac{1}{2}\mathrm{Re}(\gamma)-r-2}\right),
\end{align*}

\noindent where the implied constants depend on $\delta$, but are uniform in $\gamma$ (when restricted to a compact set). This bound can be extended to points $w$ close to $\Gamma_R$ with the standard contour deformation argument -- see \cite[Corollary 7.9]{dkmlvz}. To conclude the proof of \eqref{eq:R1},  note that we have from \eqref{eq:Rsol2} that $\lim_{w\to\infty}w(R(w)-I)=-\int_{\Gamma_R}R_-(s)\Delta_R(s)\frac{ds}{2\pi i}$, for which repeating our previous argument shows the claim.

We now move onto the proof of \eqref{eq:R2}. Here our goal is to show that $R(w)$ is an analytic in $\gamma$ on the set $\lbrace\gamma\in \C:\mathrm{Re}(\gamma)>-2\rbrace$. Then Cauchy's integral formula combined with \eqref{eq:R1} will give \eqref{eq:R2}. We note that going back in our chain of transformations, the existence of $R$ lets us define the matrix $Y$ in terms of $R$, the parametrices, and our transformations. Moreover, the RHP for $R$ induces a RHP for $Y$ as well and this RHP is precisely the one appearing in Lemma \ref{le:RHP}, though checking the asymptotic behavior at $x$ is not completely obvious. For this, we note first that reversing our transformations, $T$ and $Y$ have the same behavior at $x$ so it's enough to study asymptotics of $T$. For this, we note that if $w\to x$ and $w\in \mathrm{Int}(\Sigma)$, a direct calculation (using the definitions of our transformation, the definition of $h_r$, the definition of $Q_r$, and the definition of the incomplete gamma function) shows that we have 

\begin{align*}
T(w)&=R(w)\begin{pmatrix}
w^{\gamma/2}(w-x)^{-\gamma/2}e^{\zeta(w)}-e^{-kxw}(h_r(w;\gamma)+Q_r(w)) & e^{kxw}\\
-e^{-kxw} & 0
\end{pmatrix}\\
&=\begin{pmatrix}
\mathcal{O}(1) & \mathcal{O}(1)\\
\mathcal{O}(1) & \mathcal{O}(1)
\end{pmatrix}\begin{pmatrix}
\mathcal{O}(1)+w^{\gamma/2}(w-x)^{-\gamma/2}e^{\zeta(w)}\left(1-\frac{\Gamma(\frac{\gamma}{2},\zeta(w))}{\Gamma(\frac{\gamma}{2})}\right) & \mathcal{O}(1)\\
\mathcal{O}(1) & 0
\end{pmatrix}\\
&=\mathcal{O}(1)
\end{align*}

\noindent so we see that as $w\to x$ from $\mathrm{Int}(\Sigma)$, $Y(w)=\mathcal{O}(1)$. On the other hand, a similar argument shows that as $w\to x$ from $\mathrm{Ext}(\Sigma)$, we have 

\begin{align*}
T(w)&=R(w)\begin{pmatrix}
w^{\gamma/2}(w-x)^{-\gamma/2}-\frac{Q_r(w)+h_r(w)}{e^{kxw}e^{\zeta(w)}} & w^{-\gamma/2}(w-x)^{\gamma/2} (Q_r(w)+h_r(w))\\
-e^{-kxw}e^{-\zeta(w)} & w^{-\gamma/2}(w-x)^{\gamma/2}
\end{pmatrix}\\
&=\begin{pmatrix}
\mathcal{O}(1) & \mathcal{O}(1)\\
\mathcal{O}(1) & \mathcal{O}(1)
\end{pmatrix}\begin{pmatrix}
\mathcal{O}(1) & \mathcal{O}(1)+\mathcal{O}(|w-x|^{\frac{\mathrm{Re}(\gamma)}{2}})\\
\mathcal{O}(1) & \mathcal{O}(|w-x|^{\frac{\mathrm{Re}(\gamma)}{2}})\\
\end{pmatrix}\\
&=\begin{pmatrix}
\mathcal{O}(1) & \mathcal{O}(1)+\mathcal{O}(|w-x|^{\frac{\mathrm{Re}(\gamma)}{2}})\\
\mathcal{O}(1) & \mathcal{O}(1)+\mathcal{O}(|w-x|^{\frac{\mathrm{Re}(\gamma)}{2}})
\end{pmatrix}.
\end{align*}

\noindent We conclude that $Y$ defined from $R$ satisfies the asymptotic behavior \eqref{eq:Yatxasy1}.

It is then another standard argument (using the jump condition of $Y$, its asymptotic behavior, Liouville's theorem, and some regularity properties of the Cauchy transform -- we omit the details) that the polynomials $p_{N+k}(w)$ and $q_{N+k-1}(w^{-1})$ must exist and $Y_{N+k}$ is given by \eqref{eq:Ydef} in terms of these polynomials. More precisely,  one has 

\begin{equation*}
\frac{1}{\chi_{N+k}}p_{N+k}(w)=\frac{1}{\widehat{D}_{N+k-1}}\begin{vmatrix}
\oint_{\Sigma} f(s)\frac{ds}{2\pi i s} & \cdots & \oint_{\Sigma} s^{N+k}f(s)\frac{ds}{2\pi i s}\\
\oint_{\Sigma} s^{-1}f(s)\frac{ds}{2\pi i s} & \cdots & \oint_{\Sigma} s^{N+k-1}f(s)\frac{ds}{2\pi i s}\\
\vdots &  & \vdots\\
\oint_{\Sigma} s^{-N-k+1}f(s)\frac{ds}{2\pi i s} & \cdots  &\oint_{\Sigma} s^1 f(s)\frac{ds}{2\pi i s}\\
1 & \cdots & w^{N+k}
\end{vmatrix},
\end{equation*}

\noindent where $\widehat{D}_{N+k-1}=\det(\oint_{\Sigma}s^{i-j} f(s)\frac{ds}{2\pi i s})_{i,j=0}^{N+k-1}$ (as in Lemma \ref{le:ortoq}) and a similar expression exists for $Y_{21}(w)$, namely it equals the polynomial $-\chi_{N+k-1}w^{N+k-1}q_{N+k-1}(w^{-1})$. In particular, the uniqueness of the solution to the $R$-RHP, which then implies the uniqueness of the solution to $Y$-RHP guarantees that $\mathcal{D}_{N+k-1}\neq 0$. Now all of the entries appearing in this determinant as well as $\mathcal{D}_{N+k-1}$ are analytic functions of $\gamma$ so we conclude that $Y_{11}$ (and similarly other entries of $Y$) are analytic functions of $\gamma$. Then, going back to $R$, we conclude that $R$ is an analytic function of $\gamma$.

Now to obtain \eqref{eq:R2}, we write for a fixed $\gamma$ with $\mathrm{Re}(\gamma)>-2$, $L_\gamma$ for a square of side length $\epsilon$ centered at $\gamma$ (epsilon less than the distance to the boundary of the set). Let us write also $R(w,\gamma)$ to highlight the dependence on $\gamma$. We note that by analyticity (Cauchy's integral formula), we have

\begin{align*}
\partial_\gamma R(w,\gamma)=\frac{1}{2\pi i}\oint_{L_\gamma}\frac{R(w,\mu)}{(\mu-\gamma)^2}\frac{d\mu}{2\pi i}=\frac{1}{2\pi i}\oint_{L_\gamma}\frac{R(w,\mu)-I}{(\mu-\gamma)^2}\frac{d\mu}{2\pi i}.
\end{align*}

\noindent The first estimate in \eqref{eq:R2} then follows from the first estimate in \eqref{eq:R1}. The second claim is similar and uses again the expression $\lim_{w\to\infty}(w(R(w,\gamma)-I))=-\oint_{\Gamma_R}R_-(s,\gamma)\Delta_R(s,\gamma)\frac{ds}{2\pi i}$. For this, we also need an estimate for $\partial_\gamma \Delta_R(s,\gamma)$. This also can be estimated with a similar Cauchy integral formula argument due to the analyticity in $\gamma$, and the claim follows from our bounds on $\Delta_R(s,\gamma)$. This concludes the proof.
\end{proof}

\end{document}